\documentclass[a4paper, 10pt, parskip=half]{scrartcl}

\usepackage[utf8]{inputenc}
\usepackage[T1]{fontenc}
\usepackage{lmodern}
\usepackage{amsmath}
\usepackage{amssymb}
\usepackage{amsthm}
\usepackage{amsfonts}
\usepackage{dsfont}
\usepackage{mathtools}
\usepackage{hyperref}
\hypersetup{
  colorlinks=true,
  linkcolor=black,
  citecolor=black,
  urlcolor=blue,
  pdftitle={Shrinking vs. expanding: the evolution of spatial support in degenerate Keller--Segel systems},
  pdfauthor={Mario Fuest, Frederic Heihoff},
  pdfkeywords={degenerate diffusion; support shrinking; finite speed of propagation; critical parameters; chemotaxis},
  bookmarksopen=true,
}

\usepackage[numbers]{natbib}

\RequirePackage{geometry}
\newcommand{\R}{\mathbb{R}}

\newcommand{\N}{\mathbb{N}}

\newcommand{\mc}[1]{\mathcal{#1}}

\ifdefined\labelenumi
  \renewcommand{\labelenumi}{(\roman{enumi})}
  
\fi

\newcommand{\eps}{\varepsilon}

\DeclareMathOperator{\supp}{supp}
\DeclareMathOperator{\ess}{ess}

\newcommand{\defs}{\coloneqq}
\newcommand{\sfed}{\eqqcolon}

\newcommand{\nea}{\nearrow}

\newcommand{\sea}{\searrow}

\newcommand{\ol}{\overline}
\newcommand{\ul}{\underline}

\newcommand{\dx}{\,\mathrm{d}x}

\newcommand{\drho}{\,\mathrm{d}\rho}

\newcommand{\loc}{\mathrm{loc}}

\newcommand{\texist}{T_0}

\newcommand{\intom}{\int_\Omega}
\newcommand{\intnt}{\int_0^T}
\newcommand{\intntom}{\intnt \intom}

\newcommand{\Ombar}{\ol \Omega}

\newcommand{\leb}[2][\Omega]{\ensuremath{L^{#2}(#1)}}

\newcommand{\sob}[3][\Omega]{\ensuremath{W^{#2, #3}(#1)}}

\newcommand{\grad}{\nabla}
\newcommand{\laplace}{\Delta}

\newcommand{\ue}{u_\eps}
\newcommand{\uet}{u_{\eps t}}
\newcommand{\ve}{v_\eps}

\renewcommand{\L}[1]{{L^{#1}(\Omega)}}
\newcommand\numberthis{\addtocounter{equation}{1}\tag{\theequation}}

\newcommand{\Peps}{\mc P_\eps}

\newcommand{\we}{w_\eps}

\newcommand{\Acrit}{A_{\mathrm{crit}}}
\newcommand{\Ccrit}{C_{\mathrm{crit}}}

\newcommand{\rfn}{\rho}
\newcommand{\wsub}{\ul w}
\newcommand{\wsubmid}{\ul w_{\mathrm{mid}}}
\newcommand{\wsubout}{\ul w_{\mathrm{out}}}
\newcommand{\Tsub}{T\raisebox{4pt}{\scriptsize $\star$}}
\newcommand{\Asub}{C\raisebox{4pt}{\scriptsize $\star$}} % so it does not look like a different symbol in fractions

\newcommand{\wsup}{\ol w}
\newcommand{\wsupmid}{\ol w_{\mathrm{mid}}}
\newcommand{\wsupout}{\ol w_{\mathrm{out}}}
\newcommand{\Tsup}{\Tsub}
\newcommand{\Asup}{\Asub}

\makeatletter
\renewenvironment{proof}[1][\proofname]{\par
  \pushQED{\qed}%
  \normalfont \topsep0\p@\relax
  \trivlist
  \item[\hskip\labelsep\scshape
  #1\@addpunct{.}]\ignorespaces
}{%
  \popQED\endtrivlist\@endpefalse
}
\makeatother

\makeatletter
\renewcommand{\paragraph}{%
  \@startsection{paragraph}{4}%
  {\z@}{1ex \@plus 1ex \@minus .2ex}{-1em}%
  {\normalfont\normalsize\bfseries}%
}
\makeatother

\newtheoremstyle{nremark}
  {\topsep} % Space above
  {\topsep} % Space below
  {} % Body font
  {} % Indent amount
  {\bfseries} % Theorem head font
  {.} % Punctuation after theorem head
  {.5em} % Space after theorem head
  {} % Theorem head spec (can be left empty, meaning `normal')

\newtheorem{base}{Base}[section]
\numberwithin{equation}{section}

\newtheorem{theorem}[base]{Theorem} \newtheorem*{theorem*}{Theroem}
\newtheorem{lemma}[base]{Lemma} \newtheorem*{lemma*}{Lemma}
 \newtheorem*{prop*}{Proposition}
 \newtheorem*{cor*}{Corollary}

\theoremstyle{nremark}
\newtheorem{definition}[base]{Definition} \newtheorem*{definition*}{Definition}
 \newtheorem*{example*}{Example}
 \newtheorem*{cond*}{Condition}

\newtheorem{remark}[base]{Remark} \newtheorem*{remark*}{Remark}

\setkomafont{title}{\normalfont \Large}
\title{Shrinking vs.\ expanding: the evolution of spatial support in degenerate Keller--Segel systems}

\usepackage{authblk}

\author[1]{Mario Fuest\footnote{e-mail: fuest@ifam.uni-hannover.de, ORCID: 0000-0002-8471-4451}}
\author[2]{Frederic Heihoff\footnote{e-mail: fheihoff@math.upb.de, ORCID: 0000-0003-3654-0271, corresponding author}}
\affil[1]{Leibniz Universität Hannover, Institut für Angewandte Mathematik, Welfengarten 1, 30167 Hannover, Germany}
\affil[2]{Institut für Mathematik, Universität Paderborn, Warburger Str.~100, 33098 Paderborn, Germany}  
\date{}

\begin{document}
\maketitle

\KOMAoptions{abstract=true}
\begin{abstract}
\noindent
We consider radially symmetric solutions of the degenerate Keller--Segel system
\begin{align*}
  \begin{cases}
    \partial_t u=\nabla\cdot (u^{m-1}\nabla u - u\nabla v),\\
    0=\Delta v -\mu +u,\quad\mu =\frac{1}{|\Omega|}\int_\Omega u,
  \end{cases}
\end{align*}
in balls $\Omega\subset\mathbb R^n$, $n\ge 1$, where $m>1$ is arbitrary.
Our main result states that the initial evolution of the positivity set of $u$
is essentially determined by the shape of the (nonnegative, radially symmetric, Hölder continuous) initial data $u_0$ near the boundary of its support $\overline{B_{r_1}(0)}\subsetneq\Omega$:
It shrinks for sufficiently flat and expands for sufficiently steep $u_0$.
\\[2pt]
More precisely, there exists an explicit constant $A_{\mathrm{crit}} \in (0, \infty)$ (depending only on $m, n, R, r_1$ and $\int_\Omega u_0$) such that if
\begin{align*}
  u_0(x)\le A(r_1-|x|)^\frac{1}{m-1}
  \qquad \text{for all $|x|\in(r_0, r_1)$ and some $r_0\in(0,r_1)$ and $A<A_{\mathrm{crit}}$},
\end{align*}
then there are $T>0$ and $\zeta>0$ such that $\sup\{\, |x| \mid x \in \operatorname{supp} u(\cdot, t)\,\}\le r_1 -\zeta t$ for all $t\in(0, T)$, 
while if
\begin{align*}
  u_0(x)\ge A(r_1-|x|)^\frac{1}{m-1}
  \qquad \text{for all $|x|\in(r_0, r_1)$ and some $r_0 \in (0, r_1)$ and $A>A_{\mathrm{crit}}$},
\end{align*}
then we can find $T>0$ and $\zeta>0$ such that $\sup\{\, |x| \mid x \in \operatorname{supp} u(\cdot, t)\,\}\ge r_1 +\zeta t$ for all $t\in(0, T)$.
\\[5pt]
% 35B33 Critical exponents in context of PDEs
% 35B51 Comparison principles in context of PDEs
% 35K59 Quasilinear parabolic equations
% 35K65 Degenerate parabolic equations 
% 92C17 Cell movement (chemotaxis, etc.) 
\end{abstract}

\section{Introduction}
The present paper studies the evolution of the positivity set of solutions to a model problem with degenerate diffusion and attractive taxis, namely
\begin{align}\label{prob:main}
  \begin{cases}
    \partial_t u = \nabla \cdot (u^{m-1} \nabla u - u \nabla v)          & \text{in $\Omega \times (0, T)$}, \\
    0   = \Delta v - \mu + u, \quad \mu = \frac{1}{|\Omega|} \intom u_0, & \text{in $\Omega \times (0, T)$}, \\
    \partial_\nu u = \partial_\nu v = 0                                  & \text{on $\partial \Omega \times (0, T)$}, \\
    u(\cdot, 0) = u_0                                                    & \text{in $\Omega$},
  \end{cases}
\end{align}
where $m > 1$ is a given parameter and $\Omega \subset \R^n$ is a ball.
The system \eqref{prob:main} is a parabolic--elliptic simplification of the celebrated Keller--Segel model (\cite{KellerSegelInitiationSlimeMold1970}),
which models the spatio-temporal evolution of an organism (with density $u$) that is partially attracted by a chemical (with concentration $v$) produced by itself.
This chemotactic effect is modelled by the term $- \nabla \cdot (u \nabla v)$
and the non-directed mobility is supposed to be density-dependent and to take the form $\nabla \cdot (u^{m-1} \nabla u)$;
see, e.g., \cite{PainterHillenVolumefillingQuorumsensingModels2002}, \cite{HillenPainterUserGuidePDE2009}, \cite{GurtinMacCamyDiffusionBiologicalPopulations1977} for modelling considerations.

With respect to blow-up in \eqref{prob:main}, the exponent $m_c \defs 2 - \frac2n$ is critical:
If $m > m_c$, then for all sufficiently regular, nonnegative initial data, there exist global, bounded solutions
(\cite{SugiyamaGlobalExistenceSubcritical2006}, %GE, PE, dengerate, q=1, full space
see also \cite{KowalczykSzymanskaGlobalExistenceSolutions2008}, % GE, PP, degenerate, q=1, bounded domain
\cite{IshidaEtAlBoundednessQuasilinearKeller2014}), % GE, PP, degenerate, S(u), bounded (not necessarily convex) domain
while if $m \in (1, m_c)$, then for all $M > 0$ one can find nonnegative $u_0$ with $\intom u_0 = M$ and a solution of \eqref{prob:main} blowing up in finite time
(\cite{CieslakWinklerFinitetimeBlowupQuasilinear2008}, % FTBU, PE, degenerate, q=1, ball. GE for nondegenerate
\cite{SugiyamaGlobalExistenceSubcritical2006}, %FTBU, PE, dengerate, q=1, full space
see also \cite{IshidaYokotaBlowupFiniteInfinite2013}, % blow up (FTBU or IFTBU), PP, degnerate, u^{q-1}, ball
\cite{HashiraEtAlFinitetimeBlowupQuasilinear2018}). % FTBU, PP, degnerate, u^{q-1}, ball
Moreover, in the fully parabolic full-space setting with $m = m_c > 1$,
there is a critical mass distinguishing between global boundedness and the possibility of finite-time blow-up (\cite{BlanchetLaurencotParabolicparabolicKellerSegelSystem2013}, \cite{LaurencotMizoguchiFiniteTimeBlowup2017}).
Regarding the fully parabolic nondegenerate setting with potential nonlinear taxis sensitivity,
we refer to \cite{HorstmannWinklerBoundednessVsBlowup2005}, \cite{TaoWinklerBoundednessQuasilinearParabolic2012}, \cite{IshidaEtAlBoundednessQuasilinearKeller2014} for global boundedness,
to \cite{WinklerDoesVolumefillingEffect2009} for blow-up in either finite or infinite time, 
and to \cite{CaoFuestFinitetimeBlowfullyParabolic2024} (as well as to the precedents \cite{CieslakStinnerFinitetimeBlowupGlobalintime2012}, \cite{CieslakStinnerFiniteTimeBlowupSupercritical2014}, \cite{CieslakStinnerNewCriticalExponents2015}) for finite-time blow-up.
For an overview of further dichotomies between boundedness and blow-up for chemotaxis systems, see the surveys \cite{BellomoEtAlMathematicalTheoryKeller2015}, \cite{LankeitWinklerFacingLowRegularity2019}.
 
Much less studied are support propagation properties of \eqref{prob:main} --
in contrast to those of its taxis-free relative 
\begin{align}\label{prob:pme}
  \partial_t u = \nabla \cdot (u^{m-1} \nabla u),
\end{align}
the porous medium equation (PME) (where again $m > 1$ is a given parameter),
for which finite speed of propagation and the existence of waiting times constitute celebrated phenomena.
The former states that the positivity set of solutions to nonnegative, nontrivial initial $u_0$ with compact support does not grow infinitely fast (\cite[Theorem~14.6]{VazquezPorousMediumEquation2006}).
This contrasts the nondegenerate heat equation, \eqref{prob:pme} with $m = 1$,
whose solutions to such initial data become immediately positive by the strict maximum principle.

While solutions to the PME eventually propagate to all compact subsets of the domain in finite time (\cite[Theorem~14.3]{VazquezPorousMediumEquation2006}),
the question whether the support already needs to grow near $t=0$ is more delicate.
As it turns out, the answer is completely determined by the flatness of $u_0$ near a point $x_0 \in \supp u_0$:
If $\limsup_{r \sea 0} r^{-\frac{2}{m-1}-n} \int_{B_r(x_0)} u_0(x) \dx = \infty$,
% that's the formula from \cite[(7)]{DalPassoEtAlWaitingTimePhenomena2003}, \cite[Propsition~4.2]{AlikakosPointwiseBehaviorSolutions1985} seems to contain a misprint
the support near $x_0$ immediately expands, while, if this quantity is finite,
then there is a so-called waiting time upon which the support near $x_0$ stays constant before it starts to grow, see \cite[Propsition~4.2]{AlikakosPointwiseBehaviorSolutions1985}.
In particular, if the initial data are of the form $u_0(x) = C (|x|-r_1)_+^\alpha$ for $|x|$ close to $r_1$, a waiting time phenomenon near $x_0$ occurs if and only if $\alpha \ge \frac{2}{m-1}$.
Similar results have been obtained for other degenerate equations as well,
see for instance \cite{DalPassoEtAlWaitingTimePhenomena2003} both for doubly nonlinear and for higher order equations
and \cite{LaurencotMatiocFiniteSpeedPropagation2017} for a repulsive--repulsive fully cross-diffusive system.

On the other hand, degenerate diffusion as in \eqref{prob:pme} alone cannot cause a retraction of the free boundary;
the size of the support can never strictly decrease (\cite[Proposition~14.1]{VazquezPorousMediumEquation2006}).
This raises the question whether an additional mechanism added to the PME can lead to initial support shrinking.
A natural such candidate is attractive taxis:
As discussed above, the thereby introduced tendency towards aggregation is in some situations sufficiently strong to completely overcome the stabilizing effect of diffusion in the sense that it causes (finite-time) blow-up,
so it is conceivable that it may also reverse the direction of movement of the free boundary, at least for some time.

This is indeed the case: Initial support shrinking has been detected in \cite{FischerAdvectiondrivenSupportShrinking2013}
for a chemotaxis system with prevention of overcrowding introduced in \cite{BurgerEtAlKellerSegelModelChemotaxis2006},
and in \cite{XuEtAlChemotaxisModelDegenerate2020} for a fully parabolic chemotaxis--consumption model (where one then also needs to impose certain conditions on $v(\cdot, 0)$).
On the other hand, \cite{BlackAbsenceDeadcoreFormations2024} shows that solutions to degenerate chemotaxis systems may not form dead-cores in finite time, i.e., that they stay positive for positive initial data.
 
As to \eqref{prob:main}, it is known that the support propagates with (at most) finite speed
if either $\Omega = \R$ (\cite{SugiyamaFiniteSpeedPropagation2012}) or $\Omega = \R^n$ (\cite{KimYaoPatlakKellerSegelModelIts2012}),
that the positivity set may be contained in some proper subset of the domain (\cite{StevensWinklerTaxisdrivenPersistentLocalization2022}, \cite{KimYaoPatlakKellerSegelModelIts2012})
and that initial support shrinking is possible if $\Omega = \R^n$ and $m > 2 - \frac2n$ (\cite{FischerAdvectiondrivenSupportShrinking2013}).

\paragraph{Main result.}
Our main result goes beyond these findings and identifies a critical condition distinguishing (for a wide class of initial data) between initial inward and outward motion of the free boundary
for radially symmetric solutions of \eqref{prob:main}.
\begin{theorem}\label{th:main}
  Let $n \in \N$, $R > 0$, $\Omega = B_R(0) \subset \R^n$, $m > 1$ and $r_1 \in (0, R)$.
  Moreover, let $u_0 \in \leb\infty$ be radially symmetric and nonnegative a.e.\ with $\sup (\ess\supp u_0) = r_1$
  and set
  \begin{align}\label{eq:def_c_crit}
    \Acrit \defs \left[ \frac{\intom u_0 \left(1 - \frac{r_1^n}{R^n}\right) (m-1)}{\omega_n r_1^{n-1}} \right]^\frac{1}{m-1}.
  \end{align}
  Then there exist $T_0 \in (0, \infty]$ and a nonnegative, radially symmetric weak solution $(u, v)$ of \eqref{prob:main} in $\Ombar \times [0, T_0)$ in the sense of Definition~\ref{def:weak_sol} below
  (which is moreover Hölder continuous if $u_0$ is) such that
  if
  \begin{align}\label{eq:cond_shrinking}
    \exists A < \Acrit\; \exists r_0 \in (0, r_1)\; \forall r \in (r_0, r_1) : u_0(r) \le A (r_1-r)^\frac{1}{m-1},
  \end{align}
  then the spatial support initially shrinks in the sense that there are $T \in (0,T_0)$ and $\zeta > 0$ with
  \begin{align}\label{eq:result_shrinking}
    \sup(\ess\supp u(\cdot, t)) \leq r_1 - \zeta t
    \qquad \text{for a.e.\ $t \in (0, T)$,}
  \end{align}
  while if
  \begin{align}\label{eq:cond_expanding}
    \exists A > \Acrit\; \exists r_0 \in (0, r_1)\; \forall r \in (r_0, r_1) : u_0(r) \ge A (r_1-r)^\frac{1}{m-1},
  \end{align}
  then the spatial support initially expands in the sense that there are $T \in (0,T_0)$ and $\zeta > 0$ with
  \begin{align}\label{eq:result_expanding}
    \sup (\ess \supp u(\cdot, t) ) \geq r_1 + \zeta t  
    \qquad \text{for a.e.\ $t \in (0, T)$.}
  \end{align}
  (Here and throughout the article, we write $\varphi(|x|)$ for $\varphi(x)$ whenever $\varphi$ is radially symmetric.)
\end{theorem}

\begin{remark}
  Since $(r_1-r_0)^{\alpha-\frac{1}{m-1}}$ converges to $0$ (respectively, $\infty$) as $r_0 \nea r_1$ if $\alpha > \frac{1}{m-1}$ (respectively, $\alpha < \frac{1}{m-1}$),
  we see that
  \begin{align}\label{eq:cond_shrinking_alpha}
    \exists \alpha > \frac{1}{m-1}\; \exists A > 0\; \exists r_0' \in (0, r_1)\; \forall r \in (r_0', r_1) : u_0(r) \le A (r_1 - r)^\alpha
  \end{align}
  and
  \begin{align}\label{eq:cond_expanding_alpha}
    \exists \alpha < \frac{1}{m-1}\; \exists A > 0\; \exists r_0' \in (0, r_1)\; \forall r \in (r_0', r_1) : u_0(r) \ge A (r_1 - r)^\alpha
  \end{align}
  imply \eqref{eq:cond_shrinking} and \eqref{eq:cond_expanding}, respectively, and hence initial shrinking and initial expanding, respectively.
\end{remark}

\begin{remark}
  Our strategy of proof, to be outlined below, necessitates that we limit our attention to radially symmetric settings.
  Moreover, we note that we can only treat the most outward boundary points of the support; Theorem~\ref{th:main} requires $u_0(x) = 0$ for all $|x| \ge r_1$.
  %That our methods are better suited to handle such points will also be discussed below.
\end{remark}

\begin{remark}
  Let us compare Theorem~\ref{th:main} to results regarding the PME \eqref{prob:pme}
  beyond the fact that the latter are not limited to radially symmetric settings and specific boundary points of the initial positivity set.
  \begin{enumerate}
    \item 
      The critical exponent $\frac{1}{m-1}$ in Theorem~\ref{th:main} is smaller than the critical waiting time exponent $\frac{2}{m-1}$ for the PME.

    \item
      While the condition \eqref{eq:cond_expanding_alpha} with $\frac{1}{m-1}$ replaced by $\frac{2}{m-1}$ implies initial expanding of the support for solutions to the PME,
      the analogue of \eqref{eq:cond_shrinking_alpha} does not imply initial support shrinking but instead that the support remains constant for some time.
      That is, the critical exponent for the PME distinguishes between the existence of waiting times and immediate expanding,
      the critical exponent for \eqref{prob:main} between shrinking and expanding.

    \item
      The behaviour at the critical exponent is also different.
      While for \eqref{prob:main}, Theorem~\ref{th:main} shows the existence of a critical parameter $\Acrit$
      and hence in particular that the support of solutions to initial data $u_0$ with $u_0(r) \sim (r_1-r)_+^\frac{1}{m-1}$ may shrink or expand depending on the implied constant,
      all solutions to the PME with initial data $u_0(r) \le A(r_1-r)_+^\frac{2}{m-1}$ exhibit a waiting time phenomenon, regardless of how large $A$ is
      (\cite[Propsition~4.2]{AlikakosPointwiseBehaviorSolutions1985}).
  \end{enumerate}
\end{remark}

\begin{remark}
  For the variant of \eqref{prob:main} posed in the full space and with the second equation replaced by $-\Delta v = u$,
  it has been claimed (without detailed proof) in \cite[Corollary on p.~1613]{FischerAdvectiondrivenSupportShrinking2013} that \eqref{eq:cond_shrinking_alpha} implies initial shrinking,
  provided that $m > 2 - \frac2n$.
  We again note a few difference to the situation in Theorem~\ref{th:main}.
  \begin{enumerate}
    \item
      The results in \cite{FischerAdvectiondrivenSupportShrinking2013} also hold in nonradial settings, while we only consider the radially symmetric case.
      As noted there, the geometric conditions on a point $x_0 \in \partial \supp u_0$ required by \cite{FischerAdvectiondrivenSupportShrinking2013}
      are in particular fulfilled whenever $x_0$ belongs to the boundary of the convex hull of $\supp u_0$.
      If $u_0$ is radially symmetric, the latter assumption is equivalent to requiring $u_0(x) = 0$ for all $|x| \ge |x_0|$, as we do in Theorem~\ref{th:main}.

    \item
      Since \cite{FischerAdvectiondrivenSupportShrinking2013} exclusively studies under which conditions shrinking occurs -- and not also when expansion happens --,
      optimality of the exponent $\frac{1}{m-1}$ is not obtained there.

    \item
      In \cite{FischerAdvectiondrivenSupportShrinking2013}, shrinking is not claimed for the critical exponent (cf.\ (43) in \cite{FischerAdvectiondrivenSupportShrinking2013}).
      However, although Theorem~\ref{th:main} only covers balls with finite radius, we note that the value $\Acrit$ in \eqref{eq:def_c_crit} converges to a positive, finite number as $R \to \infty$.
      This indicates that also in the full space setting shrinking should be possible for initial data fulfilling $u_0(r) \sim (r_1-r)_+^\frac{1}{m-1}$ near $r_1$,
      and that hence the criterion (43) in \cite{FischerAdvectiondrivenSupportShrinking2013} may not be optimal.
  
    \item
      Finally, we emphasize that, unlike \cite{FischerAdvectiondrivenSupportShrinking2013}, Theorem~\ref{th:main} does not pose any assumption on $m$ beyond the degeneracy condition $m > 1$;
      our main result in particular holds in cases where bounded weak solutions cease to exist after finite time (which may be the case if $m < 2 - \frac2n$,
      see \cite{CieslakWinklerFinitetimeBlowupQuasilinear2008}, \cite{HashiraEtAlFinitetimeBlowupQuasilinear2018}).
  \end{enumerate}
\end{remark}

\paragraph{Strategy of our proof.}
The radial symmetry assumption implies that the transformed quantity $w(s, t) \defs n \int_0^{s^\frac1n} \rho^{n-1} u(\rho, t) \drho$, first introduced in \cite{JagerLuckhausExplosionsSolutionsSystem1992},
solves the scalar equation
\begin{align}\label{eq:intro:w_eq}
  \partial_t w - n^2 s^{2-\frac2n} (\partial_s w)^{m-1} \partial_{ss} w - w \partial_s w + \mu s \partial_s w = 0
  \qquad \text{in $(0, R^n) \times (0, T_0)$}
\end{align}
with $w(0, t) = 0$ and $w(R^n, t) = \mu R^n$ for all $t \in (0, T)$.
Since
\begin{align}\label{eq:intro:inf_w}
  \big( \sup \{\, r \in (0, R) \mid u(r, t) > 0\,\}\big)^n = \inf \{\, s \in (0, R^n) \mid w(s, t) = \mu R^n\,\} \sfed I_w(t)
  \qquad \text{for all $t \in (0, T_0)$},
\end{align}
the evolution of the most outwards boundary point of the spatial support of $u$ corresponds to monotonicity properties of the infimum $I_w(t)$ in \eqref{eq:intro:inf_w}.
(Potential more inward interface points cannot be characterized in such a way, which is the reason that Theorem~\ref{th:main} does not cover them.)

In order to derive suitable estimates for $I_w(t)$, we shall make use of a comparison principle (cf.\ Lemma~\ref{lm:cmp_princ}).
For possible sub- and supersolutions, we take the ansatz
\begin{equation}\label{eq:intro:def_w_cmp}
  w(s,t) \defs \begin{cases}
    \mu R^n - C(r_1^n + \theta t - s)^{\frac{m}{m-1}} & \text{if } s \le r_1^n + \theta t, \\
    \mu R^n                                           & \text{if } s > r_1^n + \theta t,
  \end{cases}
\end{equation}
where $\theta$ is negative in the shrinking and positive in the expanding case.
(We note that \cite{StevensWinklerTaxisdrivenPersistentLocalization2022} utilizes such comparison functions with $\theta = 0$ to prove persistent localization.)
A direct computation reveals that, for $s < r_1^n + \theta t$, the left-hand side in \eqref{eq:intro:w_eq} equals
\begin{align}\label{eq:intro:w_cmp_calc1}
  \partial_s w \left[ -\theta + n^2 s^{2-\frac2n} \frac{(Cm)^{m-1}}{(m-1)^m} - w + \mu s\right].
\end{align}
Setting $\theta = 0$ and $s = r_1^n$, the second factor becomes
\begin{align}\label{eq:intro:w_cmp_calc2}
  n^2 r_1^{2n-2} \frac{(Cm)^{m-1}}{(m-1)^m} - \mu (R^n - r_1^n),
\end{align}
whose sign is determined by the size of $C$.
As long as $C$ is not critical, continuity arguments show that the signs of the quantities in \eqref{eq:intro:w_cmp_calc1} and \eqref{eq:intro:w_cmp_calc2} are the same
as long as $|\theta|$, $t$ and $r_1^n-s$ are sufficiently small (and $\partial_s w \geq 0$),
implying that there are parameters such that the function $w$ in \eqref{eq:intro:def_w_cmp} is a sub- or supersolution of \eqref{eq:intro:w_eq} in $(r_0^n, r_1^n) \times (0, T)$ for $r_0$ close to $r_1$ and small $T > 0$.
Since the growth conditions for $u_0$ in Theorem~\ref{th:main} translate to suitable conditions for $w(\cdot, 0)$ (cf.\ Lemma~\ref{lm:crit_w0})
and as the solution $w$ and the comparison function in \eqref{eq:intro:def_w_cmp} are correctly ordered at $r_0^n$ (and hence on the whole parabolic boundary) in a small time interval $(0, T')$ if they are initially,
the comparison principle yields the desired properties of $I_w(t)$.

As neither the solution $w$ nor the comparison functions above are (known to be) sufficiently regular to allow for rigorous applications of the comparison principle (Lemma~\ref{lm:cmp_princ}),
we instead perform these arguments for solutions $(\ue, \ve)$ to approximate, nondegenerate systems.
However, this approach introduces some additional challenges;
for instance, we need to adjust the comparison functions to account for the fact that the strict maximum principle forces $\ue$ to become positive immediately.

\paragraph{Plan of the paper.}
We recall local existence and approximation theory for \eqref{prob:main} in Section~\ref{sec:local_ex_conv}
and introduce the mass accumulation function in Section~\ref{sec:mass}.
Section~\ref{sec:cmp} not only contains the key ingredient of our proof, namely the definition of $\wsub$ and $\wsup$ and the verification of the sub- and supersolution properties, respectively,
but also the comparison theorem and the applications to $\we$.
Finally, Theorem~\ref{th:main} is proven in Section~\ref{sec:proof_main_thm}.

\paragraph{Notation.}
We henceforth fix $n \in \N$, $R > 0$, $\Omega = B_R(0) \subset \R^n$ and $m > 1$.

\section{Local existence and convergence to weak solutions}\label{sec:local_ex_conv}
To not unduly duplicate effort, we will largely rely on the weak solution theory for (\ref{prob:main}) presented in \cite{StevensWinklerTaxisdrivenPersistentLocalization2022} for the construction of the solutions discussed in Theorem~\ref{th:main}. Naturally, this means we will use the same standard notion of weak solution to (\ref{prob:main}) as in the mentioned reference.

\begin{definition}\label{def:weak_sol}
  Let $T_0 > 0$ and let $u_0 \in \leb\infty$ be nonnegative and radially symmetric with $\frac{1}{|\Omega|} \intom u_0 = \mu$.
  We call a pair of radially symmetric functions
  \begin{align*}%\label{eq:weak_sol:reg}
    (u, v) \in L_{\loc}^\infty(\Ombar \times [0, T_0)) \times L_{\loc}^\infty([0, T_0); \sob12)
    \quad \text{with} \quad
    u^m \in L_{\loc}^2([0, T_0); \sob12)
  \end{align*}
  being such that $u \ge 0$ a.e.\ in $\Omega \times (0, T_0)$
  a \emph{radial weak solution} of \eqref{prob:main} if
  \begin{align}\label{eq:weak_sol:u_eq}
      - \intntom u \varphi_t
      - \intom u_0 \varphi(\cdot, 0)
    = - \frac1m \intntom \nabla u^m \cdot \nabla \varphi
      + \intntom u \nabla v \cdot \nabla \varphi
  \end{align}
  and
  \begin{align*}%\label{eq:weak_sol:v_eq}
      \intntom \nabla v \cdot \nabla \varphi
    = - \mu \intntom \varphi
      + \intntom u \varphi
  \end{align*}
  hold for all $\varphi \in C_c^\infty(\Ombar \times [0, T_0))$.
\end{definition}
\begin{remark}\label{rm:weak_sol}
  Let $T_0 > 0$, let $u_0 \in \leb\infty$ be nonnegative and radially symmetric and let $(u, v)$ be a radial weak solution of \eqref{prob:main} in the sense of Definition~\ref{def:weak_sol}.
  \begin{enumerate}
    \item
      As also noted in \cite[Remark after Definition~2.1]{StevensWinklerTaxisdrivenPersistentLocalization2022},
      the weak formulation \eqref{eq:weak_sol:u_eq} implies that $u_t \in L_{\loc}^2([0, T_0); (\sob12)^\star)$ and hence (after redefining $u$ on a null set of times) $u \in C^0([0, T_0); \leb2)$.

    \item
      It has recently been shown in \cite[Corollary~1.9]{BlackRefiningHolderRegularity2024}
      that $u$ is locally (space-time) Hölder continuous in $\Ombar \times (0, T_0)$
      and that Hölder continuity of $u_0$ implies that $u$ is Hölder continuous up to $t=0$.
      (For $m \le 3$, the same conclusion can be obtained from the classical work \cite{PorzioVespriHolderEstimatesLocal1993}.)
  \end{enumerate}
\end{remark}

As they will play a crucial role in our coming arguments, let uns now briefly review some of the more intricate details presented in \cite{StevensWinklerTaxisdrivenPersistentLocalization2022} as part of the derivation of such weak solutions. As a matter of fact, the key ingredient used in their construction that we want to utilize is a family of approximate solutions to slightly regularized versions of the system in (\ref{prob:main}), which approach it as the approximation parameter tends to zero and consequently yield our desired solution as their limit. In particular, the central and only change to the system involves mitigation of the nonlinear degeneracy present in the diffusion operator of the first equation in (\ref{prob:main}). This not only enables the following straightforward global classical existence theory in the aforementioned reference but will also provide us with systems and corresponding classical solutions much more amenable to the comparison arguments at the core of our reasoning here.

\begin{lemma}\label{lm:local_ex_eps}
  Let $u_0 \in \leb\infty$ be nonnegative and radially symmetric.
  Then we can find $T_0 > 0$ such that for all $\eps \in (0, 1)$,
  a classical solution $(\ue, \ve)$ of 
  \begin{align}\label{prob:eps}
    \begin{cases}
      \uet = \nabla \cdot ( (\ue + \eps)^{m-1} \grad \ue - \ue \grad \ve)       & \text{in $\Omega \times (0, T_0)$}, \\
      0 = \laplace \ve - \mu + \ue, \quad \mu = \frac{1}{|\Omega|} \intom u_0,  & \text{in $\Omega \times (0, T_0)$}, \\
      \partial_\nu \ue = \partial_\nu \ve = 0                                   & \text{on $\partial\Omega \times (0, T_0)$}, \\
      \ue(\cdot, 0) = u_0                                                       & \text{in $\Omega$}
    \end{cases}
  \end{align}
  satisfying
  \begin{align*}
    \ue \in C^{2, 1}(\Ombar \times (0, T_0)) \cap C^0([0, T_0); \leb1), \quad
    \ve \in C^{2, 0}(\Ombar \times (0, T_0))
    \quad \text{and} \quad
    \intom \ve = 0
  \end{align*}
  exists, which is moreover such that $\ue(\cdot, t)$ and $\ve(\cdot, t)$ are radially symmetric for all $t \in (0, T_0)$
  and that
  \begin{align}\label{eq:local_ex_eps:bdd_u}
    0 < \ue \leq \|u_0\|_{\leb\infty} + 1
    \qquad \text{in $\Ombar \times (0, T_0)$}.
  \end{align}
  for all $\eps \in (0, 1)$.
\end{lemma}
\begin{proof}
  See \cite[Lemma~2.1]{StevensWinklerTaxisdrivenPersistentLocalization2022}.
\end{proof}

Using a series of testing procedures, sufficient convergence properties are then derived in \cite{StevensWinklerTaxisdrivenPersistentLocalization2022} for the family of approximate solutions constructed in Lemma~\ref{lm:local_ex_eps} to ensure that the weak solution properties they enjoy (due to them being classical solutions) survive the limit process in an appropriate fashion.

\begin{lemma}\label{lm:eps_sea_0}
  Let $u_0 \in \leb\infty$ be nonnegative and radially symmetric
  and let $T_0$ and $(\ue, \ve)_{\eps \in (0, 1)}$ be as given by Lemma~\ref{lm:local_ex_eps}.
  Then there exist a null sequence $(\eps_j)_{j \in \N} \subset (0, 1)$ and a radial weak solution $(u, v)$ of \eqref{prob:main} in the sense of Definition~\ref{def:weak_sol} such that
  \begin{alignat}{2}\label{eq:eps_sea_0:pw}
    \ue           & \to u           &  & \qquad \text{a.e.\ in } \Omega\times(0,T_0) \text{ as $\eps = \eps_j \sea 0$}.
  \end{alignat}
\end{lemma}
\begin{proof}
  This has been shown in \cite[Lemma~2.3]{StevensWinklerTaxisdrivenPersistentLocalization2022}.
\end{proof}

Lastly, we prove a uniform continuity result in $t = 0$ to complement the other solution properties already laid out above by another testing procedure. This will later help us to localize our central comparison arguments around the boundary of the support of the initial data.

\begin{lemma}\label{lm:uniform_convergence}
  Let $u_0 \in \leb\infty$ be nonnegative and radially symmetric
  and let $T_0$ and $(\ue, \ve)_{\eps \in (0, 1)}$ be as given by Lemma~\ref{lm:local_ex_eps}.
  Then for every $r \in (0, R)$ and $\delta > 0$, there exists $t_0 \in (0,T_0)$ such that
  \[
    \left| \int_{B_r(0)} u_0 - \int_{B_r(0)} \ue(\cdot, t) \right| \leq \delta
  \]
  for all $t\in(0,t_0)$ and $\eps \in (0,1)$.
\end{lemma}
\begin{proof}
  Using a similar testing-based approach to \cite[Lemma~2.3]{StevensWinklerTaxisdrivenPersistentLocalization2022} combined with the fundamental theorem of calculus, we can immediately see that
  \begin{align*}
    \left| \int_\Omega u_0\varphi - \intom \ue(\cdot, t) \varphi \right| &= \left| \int_0^t \int_\Omega \uet \varphi \right| \\
    %&\blue{= \left| \int_0^t\int_\Omega \left(\tfrac{1}{m}\grad(\ue + \eps)^{m} - \ue \grad \ve \right) \cdot \grad \varphi  \right|} \\
    %
    %
    &\leq \int_0^t \left( \tfrac{|\Omega|}{m} \|\ue + 1\|^m_\L{\infty} \|\laplace \varphi\|_\L{\infty} + \|\ue\|_\L{\infty} \|\grad \ve\|_\L{1} \|\grad \varphi\|_\L{\infty} \right)
  \end{align*}
  for all $\varphi \in C^\infty_c(\Omega)$, $t\in(0,T_0)$ and $\eps \in (0,1)$. By for instance testing the second equation in (\ref{prob:eps}) with $\ve$ to gain a uniform-in-time gradient bound for $\ve$ combined with the uniform bound for $\ue$ found in (\ref{eq:local_ex_eps:bdd_u}), we gain $C \equiv C(\|u_0\|_\L{\infty}) > 0$ such that 
  \begin{equation}\label{eq:uniform_convergence:test_func}
    \left| \int_\Omega u_0\varphi - \int_\Omega \ue(\cdot, t) \varphi \right| \leq C (\|\grad \varphi\|_\L{\infty} + \|\laplace \varphi\|_\L{\infty}) t
  \end{equation}
  for all $\varphi \in C^\infty_c(\Omega)$, $t\in(0,T_0)$ and $\eps \in (0,1)$.

  We now fix $r \in (0, R)$ and $\delta > 0$. We then choose a cutoff function $\varphi \in C_c^\infty(\Omega)$ with $\varphi \equiv 1$ on $B_r(0)$ and $\int_{\Omega\setminus B_r(0)} \varphi \leq \frac{\delta}{2(2\|u_0\|_\L{\infty}+1)}$. Using (\ref{eq:uniform_convergence:test_func}), we then fix $t_0 \in (0,T)$ such that 
  \[
    \left| \int_\Omega u_0\varphi - \int_\Omega \ue(\cdot, t) \varphi \right| \leq \frac{\delta}{2}
  \]
  for all $t\in(0,t_0)$ and $\eps \in (0,1)$.
  Combining these choices for $\varphi$ and $t_0$ with (\ref{eq:local_ex_eps:bdd_u}), we then gain
  \begin{align*}
    \left| \int_{B_r(0)} u_0 - \int_{B_r(0)} \ue(\cdot, t) \right| &= \left| \int_{B_r(0)} u_0\varphi - \int_{B_r(0)} \ue(\cdot, t)\varphi \right| \\
    &\leq \left| \int_{\Omega} u_0\varphi - \int_{\Omega} \ue(\cdot, t)\varphi \right| + \int_{\Omega\setminus B_r(0)} |u_0 - \ue(\cdot, t)| \varphi \\
    &\leq \frac{\delta}{2} + (2\|u_0\|_\L{\infty}+1)\int_{\Omega\setminus B_r(0)}\varphi \leq \delta
  \end{align*}
  for all $t\in(0,t_0)$ and $\eps \in (0,1)$. This completes the proof.
\end{proof}

For the remainder of the paper, we now fix a nonnegative, radially symmetric $u_0 \in \leb\infty$ and set $\mu \defs \frac{1}{|\Omega|} \intom u_0$.
Moreover, we fix the family of approximate solutions $(\ue, \ve)_{\eps \in (0,1)}$ with uniform existence time $\texist > 0$ constructed in Lemma~\ref{lm:local_ex_eps} as well as the limit solution $(u,v)$ found in Lemma~\ref{lm:eps_sea_0}.

\section{Mass accumulation functions}\label{sec:mass}
As pioneered in \cite{JagerLuckhausExplosionsSolutionsSystem1992} for systems simplified in a similar fashion to (\ref{prob:main}) and also considered in radial settings, we will from now on mostly focus on an derived quantity instead of the functions $\ue$ and $\ve$ themselves. Namely, we consider the mass accumulation functions 
\[
  \we(s,t) \defs n \int_0^{s^\frac1n} \rho^{n-1} \ue(\rho, t) \drho \quad \text{ and similarly }\quad w(s,t) \defs n \int_0^{s^\frac1n} \rho^{n-1} u(\rho, t) \drho
\]
for all $(s,t) \in [0,R^n]\times[0,\texist]$ and $\eps \in (0,1)$. Making use of radial symmetry, a straightforward computation then yields that $\we$ solves the parabolic differential equation 
\begin{equation}\label{eq:Peps_equation}
  0 = \Peps \we =  \partial_t \we
  - n^2 s^{2-\frac{2}{n}} (\partial_s \we + \eps)^{m-1} \partial_{ss} \we
  - \we \partial_s \we
  + \mu s \partial_s \we
\end{equation}
on $[0,R^n]\times[0,\texist]$ with initial data $w_0(s) \defs n \int_0^{s^\frac1n} \rho^{n-1} u_0(\rho) \drho$. Additionally, it is immediately obvious from the nonnegativity of $\ue$ as well as the mass conservation properties of (\ref{prob:eps}) that $\we(\cdot, t)$ is monotonically increasing and $\we(0, t) = 0$ as well as $\we(R^n, t) = \mu R^n$ for all $t\in[0,\texist]$.  

Let us now further note that the a.e.\ pointwise convergence properties of the family $(\ue)_{\eps\in(0,1)}$ toward $u$ in (\ref{eq:eps_sea_0:pw}) translate in a sensible way to these newly introduced quantities. In fact along the null sequence $(\eps_j)_{j \in \N}$ constructed in Lemma~\ref{lm:eps_sea_0}, we gain that 
\begin{equation}\label{eq:w_pw_conv}
  \we(s, t) \rightarrow w(s, t) \quad \text{ as } \eps = \eps_j \searrow 0
\end{equation}
for all $s\in[0, R^n]$ and a.e.\ $t\in[0,\texist]$ by an application of Lebesgue's theorem, which is possible due to the uniform bound seen in (\ref{eq:local_ex_eps:bdd_u}). Notably, this directly implies that 
\begin{equation}\label{eq:w_upper_bound}
  w(s, t) \leq \mu R^n \quad \text{ and } \quad w(R^n, t) = \mu R^n
\end{equation}
for all $s\in[0, R^n]$ and a.e.\ $t\in[0,\texist]$ due to these property uniformly holding for the approximate functions $\we$.

Next, we investigate how our central initial data conditions (\ref{eq:cond_shrinking}) and (\ref{eq:cond_expanding}) in Theorem~\ref{th:main} translate from $u_0$ to its mass accumulation counterpart $w_0$.

\begin{lemma}\label{lm:crit_w0}
  Let $u_0 \in \leb\infty$ be nonnegative and radially symmetric with $\ess\supp u_0 \subseteq \overline{B_{r_1}(0)}$ for some $r_1 \in (0, R)$. If (\ref{eq:cond_shrinking}) holds for some $A > 0$ and $r_0 \in (0,r_1)$, then 
  \begin{equation}\label{eq:cond_shrinking_w0}
    w_0(s) \geq \mu R^n - \frac{A n^{-\frac{1}{m-1}} r_0^{-\frac{(n-1)m}{m-1}} r_1^{n-1} (m-1)}{m} (r_1^n - s)^\frac{m}{m-1} \quad \text{ for all } s \in (r_0^n, r_1^n).
  \end{equation}
  Similarly, if (\ref{eq:cond_expanding}) holds for some $A > 0$ and $r_0 \in (0,r_1)$, then 
  \begin{equation}\label{eq:cond_expanding_w0}
    w_0(s) \leq \mu R^n - \frac{A n^{-\frac{1}{m-1}} r_1^{-\frac{(n-1)m}{m-1}} r_0^{n-1} (m-1)}{m} (r_1^n - s)^\frac{m}{m-1} \quad \text{ for all } s \in (r_0^n, r_1^n).
  \end{equation}
\end{lemma}
\begin{proof}
  Assuming that (\ref{eq:cond_shrinking}) holds for some $A > 0$ and $r_0 \in (0,r_1)$, a direct computation yields
  \begin{align*}
          w_0(s)
    &=    n \int_0^{s^\frac1n} \rho^{n-1} u_0(\rho) \drho
     =    n \int_0^{R} \rho^{n-1} u_0(\rho) \drho
          - n \int_{s^\frac1n}^{r_1} \rho^{n-1} u_0(\rho) \drho \\
    &\ge  \mu R^n - A n r_1^{n-1} \int_{s^\frac1n}^{r_1} (r_1 - \rho)^\frac{1}{m-1} \drho \\
    &=    \mu R^n - \frac{A n r_1^{n-1}(m-1)}{m} ((r_1^n)^\frac1n - s^\frac1n)^{\frac{m}{m-1}} \\
    &\ge  \mu R^n - \frac{A n^{-\frac{1}{m-1}}r_0^{-\frac{(n-1)m}{m-1}} r_1^{n-1} (m-1)}{m} (r_1^n - s)^{\frac{m}{m-1}}
    \qquad \text{for all $s \in (r_0^n, r_1^n)$},
  \end{align*}
  where in the last step we applied the mean value theorem, which for all $0 < a < b$ asserts the existence of $\xi \in (a, b)$ with $b^\frac1n-a^\frac1n = \frac1n \xi^{\frac{1-n}{n}} (b-a)$.

  If on the other hand (\ref{eq:cond_expanding}) holds for some $A > 0$ and $r_0 \in (0,r_1)$, (\ref{eq:cond_expanding_w0}) is obtained by an analogous argument.
\end{proof}
\begin{remark}
  We observe that the coefficients of $(r_1^n - s)^\frac{m}{m-1}$ in (\ref{eq:cond_shrinking_w0}) and (\ref{eq:cond_expanding_w0}) coincide in the limit $r_0 \nearrow r_1$ for any fixed $A > 0$. Thus if sufficiently localized to $r_1$, (\ref{eq:cond_shrinking_w0}) and (\ref{eq:cond_expanding_w0}) are still essentially as complimentary as the original conditions on $u_0$.
\end{remark}

\section{Comparison argument}\label{sec:cmp}
Our aim now is to find a (potentially short) time $T \in (0,\texist)$ and constants $C, \theta > 0$ such that $w$ stays above or below functions of the general prototype
\begin{equation}\label{eq:comp_function_proto}
  (s,t) \mapsto \begin{cases}
    \mu R^n- C (r_1^n \pm \theta t - s)^{\frac{m}{m-1}}& \text{if } s < r_1^n \pm \theta t, \\
    \mu R^n &\text{if } s \geq r_1^n \pm \theta t
  \end{cases}
\end{equation}
on $(r_0^n, R^n)$ under the assumption that such an ordering already holds true at $t = 0$ for a constant $C > 0$ either larger or smaller (depending on the case) than the critical constant 
\begin{equation}\label{eq:Ccrit_def}
  \Ccrit(r_1) \defs \frac{m-1}{m}\left(\frac{\mu(R^n - r_1^n)(m-1)}{r_1^{2n - 2} n^2}\right)^\frac{1}{m-1}.  
\end{equation}

This shall be achieved by comparing the approximate solutions $\we$ to suitably adapted versions of \eqref{eq:comp_function_proto}, which, however, are nonsmooth, so that we need to ensure the applicability of a suitable comparison principle. To this end, we will utilize the following lemma which is built on the more general comparison result seen in \cite[Lemma~5.1]{bellomoFinitetimeBlowdegenerateChemotaxis2017}. (We note that the comparison theorem presented in \cite[Lemma~2.2]{StevensWinklerTaxisdrivenPersistentLocalization2022} is also based on \cite[Lemma~5.1]{bellomoFinitetimeBlowdegenerateChemotaxis2017}, but requires the comparison functions to be strictly increasing and is hence not suitable for all of our purposes.)

\begin{lemma}\label{lm:cmp_princ}
  Let $\eps \in (0,1)$, $a, b \in \R$ with $a < b$ and $T > 0$, and suppose that $\ul w,\, \ol w \in C^0([a, b] \times [0, T)) \cap C^1((a,b) \times (0, T))$ with $\partial_s \ul w,\, \partial_s \ol w \in L_{\loc}^\infty([a,b] \times [0, T))$ 
  satisfy
  \begin{align}
    \partial_s \ul w \geq 0
    \quad \text{and} \quad
    \partial_s \ol w \geq 0
    \qquad \text{in $(a, b) \times (0, T)$} \label{eq:cmp_princ:nondec} 
  \end{align}
  and
  \begin{align}
    \ul w(\cdot, t),\, \ol w(\cdot, t) \in W_{\loc}^{2, \infty}((a, b)) \qquad \text{for all $t \in (0, T)$}. 
  \end{align}
  If
  \begin{alignat}{2}
    (\Peps \ul w)(s, t)       & \le 0 &  & \qquad \text{for a.e.\ $s \in (a, b)$ and all $t \in (0, T)$}, \label{eq:cmp_princ:subsol} \\
    (\Peps \ol w)(s, t)       & \ge 0 &  & \qquad \text{for a.e.\ $s \in (a, b)$ and all $t \in (0, T)$}, \label{eq:cmp_princ:supersol}\\
    \ul w(s, 0) - \ol w(s, 0) & \le 0 &  & \qquad \text{for all $s \in (a, b)$},  \label{eq:cmp_princ:initial_ord} \\
    \ul w(s, t) - \ol w(s, t) & \le 0 &  & \qquad \text{for all $s \in \{a, b\}$ and $t \in (0, T)$}  \label{eq:cmp_princ:boundary_ord},
  \intertext{then}
    \ul w(s, t) - \ol w(s, t) & \le 0 &  & \qquad \text{for all $s \in (a, b)$ and $t \in (0, T)$} \label{eq:cmp_princ:ord}.
  \end{alignat}
\end{lemma}

\begin{proof}
  To put our case into the framework of the more general comparison principle presented in \cite[Lemma~5.1]{bellomoFinitetimeBlowdegenerateChemotaxis2017}, we let 
  \[
    \ol W(s,t) \defs \ol w\left( \tfrac{b-a}{b}s + a, t\right) + \eps \tfrac{b-a}{b} s \quad \text{ and } \quad \ul W(s,t) \defs \ul w\left( \tfrac{b-a}{b}s + a, t\right) + \eps \tfrac{b-a}{b} s
  \]
  for all $(s, t) \in [0,L]\times[0,T)$ with $L \defs b$. As we have only modified the original functions $\ol w$, $\ul w$ in a linear fashion, the new functions $\ol W$, $\ul W$ naturally retain all of our assumed regularity properties as well as the proper ordering at the parabolic boundary of $[0,L]\times[0,T)$ due to (\ref{eq:cmp_princ:initial_ord}) as well as (\ref{eq:cmp_princ:boundary_ord}). Due to (\ref{eq:cmp_princ:nondec}) and $a < b$, it is further directly evident that $\partial_s \ol W > 0$ and $\partial_s \ul W > 0$ in $(0,L)\times(0,T)$. Finally due to (\ref{eq:cmp_princ:subsol}) and (\ref{eq:cmp_princ:supersol}), it follows by direct computation that 
  \[
    \partial_t \ul W \leq \Phi(s,t,\ul W,\partial_s \ul W, \partial_{ss} \ul W) \quad \text{ and } \quad \partial_t \ol W \geq \Phi(s,t,\ol W,\partial_s \ol W, \partial_{ss} \ol W)
  \]
  for a.e.\ $s \in (0,L)$ and all $t\in(0,T)$ with 
  \[
    \Phi(s,t,y_0,y_1,y_2) \defs n^2 \left( \tfrac{b-a}{b}s + a \right)^{2-\frac{2}{n}}\left(\tfrac{b}{b-a}\right)^{m+1} y_1^{m-1} y_2 + \left(y_0 - \eps \tfrac{b-a}{b} s - \mu(\tfrac{b-a}{b}s + a)\right)(\tfrac{b}{b-a}y_1 - \eps)
  \]
  for all $(s,t,y_0,y_1,y_2) \in G\defs (0,L)\times(0,T)\times\R\times(0,\infty)\times\R$. Moreover, we see that \[
    \frac{\partial \Phi}{\partial y_2}(s,t,y_0,y_1,y_2) \geq 0 \quad \text{ and } \quad \left|\frac{\partial \Phi}{\partial y_0}(s,t,y_0,y_1,y_2)\right| \leq \frac{b}{b-a}y_1 + 1  
  \]
  for all $(s,t,y_0,y_1,y_2) \in G$ as well as that $\frac{\partial \Phi}{\partial y_1}(\cdot, t, \cdot, \cdot, \cdot)$ is bounded on every compact set in $(0,L)\times\R\times(0,\infty)\times \R$ for all $t\in(0,T)$. Thus, our desired ordering property for $\ol W$ and $\ul W$ on $(0,L)\times(0,T)$ and therefore (\ref{eq:cmp_princ:ord}) follow from \cite[proof of Lemma~5.1]{bellomoFinitetimeBlowdegenerateChemotaxis2017} which, as already observed  in \cite[directly before Lemma~2.2]{StevensWinklerTaxisdrivenPersistentLocalization2022}, implies the conclusion of \cite[Lemma~5.1]{bellomoFinitetimeBlowdegenerateChemotaxis2017} also for the slightly weaker regularity conditions on $\ul w$ and $\ol w$ imposed here.
\end{proof}

\subsection{Case of support shrinking}
Having now established the comparison principle at the center of this section, we start treating the support shrinking case by showing that $w$ lies above a function of the type seen in (\ref{eq:comp_function_proto}) with a negative coefficient for $\theta$ under an appropriate assumption at $t = 0$. 
This will then imply $w(s, t) = \mu R^n$ and hence $u(s^\frac1n, t) = 0$ for a.e.\ sufficiently small time $t$ and $s \ge r_1^n - \theta t$;
that is, that the support indeed initially shrinks.

As already mentioned previously, our first step in this endeavour is to construct a family of subsolutions for the approximate functions $\we$ by modifying the prototypical function in (\ref{eq:comp_function_proto}) to make it compatible with the regularized equation (\ref{eq:Peps_equation}). We first note that due to the diffusion operator in our regularized system (\ref{prob:eps}) no longer being degenerate, our approximate solutions $\ue$ become immediately positive. This in turn means that $\we(\cdot, t)$ is strictly monotonically increasing. As further $\we(R^n, t) = \mu R^n$ for all $t\in[0,\texist)$, we will thus need to replace the constant extension in (\ref{eq:comp_function_proto}) by something increasing in a similar fashion to have any chance for the resulting function to be a subsolution. In fact for our approximate subsolutions, we will extend the left part of (\ref{eq:comp_function_proto}) by a linearly increasing function by moving the extension point slightly to the left of $r_1^n - \theta t$ where the first derivate is still positive. We then move the resulting function slightly down to ensure that the right boundary value is still sufficiently close to $\mu R^n$ on small time scales. Conveniently, this modification also allows us to work around the potential singularity present in the second derivative of (\ref{eq:comp_function_proto}) at $r_1^n - \theta t$. The remaining modifications and parameter choices are mostly in service of allowing us to use the optimal value for $\Ccrit$.

\begin{lemma}\label{lm:wsub_construction}
  Let $r_1 \in (0, R)$ and $\Asub \in (0, \Ccrit(r_1))$ with $\Ccrit(r_1)$ as in (\ref{eq:Ccrit_def}).
  Then there exist $r_\mathrm{min} \in (0,r_1)$, $\theta_\mathrm{max} > 0$, $\kappa > 0$ and $\eps_0 \in (0,1)$ with $2\kappa\eps_0 < \mu$ such that
  \begin{equation}\label{eq:wsub_def_delta_eta}
    \delta \defs \delta(\eps) \defs \left(\frac{\eps\kappa(m-1)}{\Asub m}\right)^{m-1}
    \quad \text{and} \quad
    \eta \defs \eta(\eps) \defs -\Asub \delta^\frac{m}{m-1} + \eps\kappa(R^n - r_1^n + \delta) 
  \end{equation}
  are such that $\eta(\eps) \geq 0$ for all $\eps \in (0, \eps_0)$,
  and that moreover the following holds for all $\eps \in (0, \eps_0)$, $r_0 \in [r_\mathrm{min}, r_1)$ and $\theta \in [0, \theta_\mathrm{max}]$:

  Let
  \begin{equation}\label{eq:wsub_definition}
    \wsub(s,t) \defs \begin{cases}
      \mu R^n - \eta - \Asub(\rfn(t) - s)^\frac{m}{m-1} \sfed \wsubmid(s,t) & \text{if }  s < \rfn(t) - \delta  \\
      \mu R^n - \eps \kappa (R^n - \theta t - s) \sfed \wsubout(s,t) & \text{if } s \geq \rfn(t) - \delta 
    \end{cases}
  \end{equation}
  for all $(s,t) \in [r_0^n, R^n]\times[0,\infty)$,
  where $\rho(t) \defs r_1^n -\theta t$ for $t \in [0,\infty)$.
  Then $\wsub \in C^1([r_0^n, R^n]\times[0,\infty))$ with $\partial_s \wsub \geq \eps\kappa > 0$ and $\wsub(\cdot,t) \in W_\loc^{2,\infty}((r_0^n, R^n))$ for all $t \in [0, \infty)$.
  If further
  \begin{equation}\label{eq:wsub_theta_condition}
    \theta \leq (\mu-2\eps_0\kappa)(R^n - r_2^n)
  \end{equation}
  for some $r_2 \in (r_1, R]$, then
  \[ 
    \Peps \wsub(s,t) \leq -\eps\kappa (R^n - r_2^n) \partial_s \wsub(s,t) \leq 0
  \]
  for all $s \in [r_0^n, r_2^n]\setminus\{\rfn(t) - \delta \}$ and $t\in[0,\infty)$.
\end{lemma}
\begin{proof}
  Due to $\Asub < \Ccrit(r_1)$, we may fix large $\kappa > 0$ as well as small $\lambda \in (0,1)$ such that 
  \begin{equation}\label{eq:wsub_C_condition}
    \Asub \leq \frac{(m-1)\kappa}{m(\kappa + 1)}\left((1-\lambda)\frac{\mu(R^n - r_1^n)(m-1)}{r_1^{2n - 2} n^2}\right)^\frac{1}{m-1}.
  \end{equation}
  We then fix small $\eps_0 \in (0,1)$ and $\theta_\mathrm{max} > 0$ as well as $r_\mathrm{min} \in (0, r_1)$ close to $r_1$ such that
  \begin{equation}\label{eq:wsub_eps_condition_1}
    2\eps_0\kappa < \lambda \mu < \mu 
  \end{equation}
  such that 
  \begin{equation}\label{eq:wsub_r_condition_with_theta}
    r_1^n - r_0^n \leq \left(\left[\frac{\lambda\mu}{2}(R^n - r_1^n) - \theta_\mathrm{max} - \eps_0\kappa (R^n - r_1^n) \right]\frac{1}{\Asub} \right)^\frac{m-1}{m},
  \end{equation}
  for all $r_0 \in [r_\mathrm{min}, r_1)$, 
  and such that the quantities $\eta$ and $\delta$ defined in \eqref{eq:wsub_def_delta_eta} fulfil
  \begin{equation}\label{eq:wsub_eps_condition_2}
    \eta \leq \frac{\lambda\mu}{2}(R^n - r_1^n)
    \quad \text{and} \quad
    \eta = \eps\left(-\Asub \left(\frac{\kappa(m-1)}{\Asub m}\right)^{m} \eps^{m-1} + \kappa (R^n - r_1^n + \delta) \right) \geq 0
  \end{equation}
  for all $\eps \in (0,\eps_0)$, which is possible as $\delta$ and thus $\eta$ converge to 0 as $\eps \searrow 0$.
  Henceforth, we fix $\eps \in (0, \eps_0)$, $r_0 \in [r_\mathrm{min}, r_1)$, $r_2 \in (r_1, R]$ and $\theta \in [0, \theta_\mathrm{max}]$.

  By straightforward calculations, we immediately see that 
  \begin{align}\label{eq:wsub_wmid_s}
    \partial_s \wsubmid(s,t)
      &= \frac{\Asub m (\rfn(t) - s)^{\frac{1}{m-1}}}{m-1}
      > 0, \\
    \partial_{ss} \wsubmid(s,t)
      &= -\frac{\Asub m (\rfn(t) - s)^{\frac{1}{m-1}-1}}{(m-1)^2}
      < 0, \label{eq:wsub_wmid_ss}\\
    \partial_t \wsubmid(s,t)
      &= -\frac{\rfn'(t)\Asub m(\rfn(t) - s)^{\frac{1}{m-1}}}{m-1}
      = \theta \partial_s \wsubmid(s,t) \label{eq:wsub_wmid_t}
  \end{align} 
  for all $s\in[r_0^n, \rfn(t) - \delta]$ and $t\in[0,\infty)$ as well as
  \begin{align*}
    \partial_s \wsubout(s,t) &= \eps\kappa, \\
    \partial_{ss} \wsubout(s,t) &= 0,\\ 
    \partial_t \wsubout(s,t) &= \eps\kappa\theta
  \end{align*}
  for all $s\in[\rfn(t) - \delta, R^n]$ and $t\in[0,\infty)$.
  Combined with \eqref{eq:wsub_def_delta_eta}, this yields
  \begin{equation*}
    \wsubmid(\rfn(t) - \delta, t) = \mu R^n - \eta - \Asub\delta^{\frac{m}{m-1}} = \mu R^n - \eps\kappa(R^n - r_1^n  + \delta) = \wsubout(\rfn(t) - \delta, t)
  \end{equation*}
  and 
  \begin{equation*}%\label{eq:wsub_ws_eq_eps}
    \partial_s \wsubmid(\rfn(t) - \delta, t) = \frac{\Asub m}{m-1} \delta^{\frac{1}{m-1}} = \eps\kappa = \partial_s \wsubout(\rfn(t) - \delta, t)
  \end{equation*}
  for all $t\in [0,\infty)$.
  Moreover, $\partial_{ss}\wsubmid$ is continuous on the compact set $[r_0^n, \rfn(t) - \delta]$ for all $t\in[0,\infty)$, 
  so that all our desired regularity properties for $\wsub$ are evidently fulfilled.
 
  We will now start the argument proper by plugging $\wsubmid$ into the functional $\Peps$ to gain 
  \[
    \Peps \wsubmid = \partial_s \wsubmid \left[ \theta - n^2 s^{2-\frac2n}(\partial_s \wsubmid + \eps)^{m-1}\frac{\partial_{ss} \wsubmid}{\partial_s \wsubmid} - \wsubmid + \mu s \right]
  \]
  for all $s\in[r_0^n, \rfn(t) - \delta)$ and $t\in[0,\infty)$.

  We first estimate the terms originating from taxis to yield 
  \begin{align*}
    - \wsubmid + \mu s &= \mu (s-R^n) + \eta + \Asub(\rfn(t) - s)^{\frac{m}{m-1}} \\
    &\leq -\mu (R^n - r_1^n) + \eta + \Asub(r_1^n - r_0^n)^{\frac{m}{m-1}} 
  \end{align*}
  for all $s\in[r_0^n, \rfn(t) - \delta)$ and $t\in[0,\infty)$.
  Applying (\ref{eq:wsub_r_condition_with_theta}) as well as (\ref{eq:wsub_eps_condition_2}) to this then results in 
  \begin{equation}\label{eq:wsub_estimate_1}
    - \wsubmid + \mu s \leq -\theta_\mathrm{max} - \eps_0 \kappa (R^n - r_1^n) -(1-\lambda)\mu(R^n - r_1^n)
  \end{equation}
  for all $s\in[r_0^n, \rfn(t) - \delta)$ and $t\in[0,\infty)$.
    
  Regarding the diffusive term, we see that \eqref{eq:wsub_wmid_s}, \eqref{eq:wsub_wmid_ss}, the fact that $\partial_s \wsubmid \ge \eps\kappa$ and \eqref{eq:wsub_C_condition} entail that
  \begin{align*}
    - n^2 s^{2-\frac2n}(\partial_s \wsubmid + \eps)^{m-1}\frac{\partial_{ss} \wsubmid}{\partial_s \wsubmid}
    &\le \frac{n^2 s^{2-\frac2n}}{m-1} \left(\frac{\kappa + 1}{\kappa}\partial_s \wsubmid\right)^{m-1} (\rfn(t) - s)^{-1} \\
    &\leq \frac{n^2 r_1^{2n - 2}}{m-1} \left(\frac{m(\kappa + 1)}{(m-1)\kappa}\Asub\right)^{m-1}
    \leq (1-\lambda)\mu(R^n - r_1^n) 
  \end{align*}
  for all $s\in[r_0^n, \rfn(t) - \delta)$ and $t\in[0,\infty)$.
  In combination with (\ref{eq:wsub_estimate_1}), this yields
  \begin{align*}
    \Peps \wsubmid &\leq \partial_s \wsubmid \left[ \theta + (1-\lambda)\mu(R^n - r_1^n) - \theta_\mathrm{max} -\eps_0\kappa(R^n - r_1^n) - (1-\lambda)\mu(R^n - r_1^n) \right] \\
    &= \partial_s \wsubmid \left[ \theta - \theta_\mathrm{max}  - \eps_0\kappa(R^n - r_1^n) \right]
    \leq -\eps\kappa (R^n - r_2^n) \partial_s \wsubmid
  \end{align*}
  for all $s\in[r_0^n, \rfn(t) - \delta)$ and $t\in[0,\infty)$.

  We now turn our attention to the linear extension $\wsubout$. As its second spatial derivative is always zero, plugging it into the functional $\Peps$ allows us to estimate as follows for all $s\in(\rfn(t) - \delta, r_2^n]$ and $t\in[0,\infty)$ by using (\ref{eq:wsub_eps_condition_1}) as well as (\ref{eq:wsub_theta_condition}):
  \begin{align*}
    \Peps \wsubout &= \partial_s \wsubout \left[ \theta - \wsubout + \mu s  \right] \\
    &=\partial_s \wsubout \left[ \theta - \mu(R^n - s) + \eps \kappa(R^n - \theta t - s) \right] \\
    &=\partial_s \wsubout \left[ \theta - (\mu-\eps \kappa)(R^n - s) - \eps \kappa \theta t \right] \\
    &\leq\partial_s \wsubout \left[ \theta - (\mu-\eps_0\kappa)(R^n - r_2^n) \right] \\
    &\leq -\eps_0\kappa\partial_s \wsubout(R^n - r_2^n) \leq -\eps\kappa\partial_s \wsubout(R^n - r_2^n)
  \end{align*}
  This completes the proof.
\end{proof}

To now prove the central result of this subsection, we will use the above family of functions for two consecutive comparison arguments on the approximate level. We begin by comparing with a stationary function $\wsub$ with $\theta = 0$ on the full interval $(r_0^n, R^n)$ to make use of the uniform boundary value of our approximate solutions at $s = R^n$. This allows us to show that the support of $w$ at the very least does not expand. We then use this result to establish a similar boundary condition at some point $r_2^n < R^n$ in the interior of the domain to facilitate a second comparison argument on $(r_0^n, r_2^n)$ with a nonstationary function $\wsub$ with $\theta > 0$, which will be key to showing that the support of $w$ in fact actually shrinks. The reason we cannot immediately compare with a nonstationary function in our construction is due to the condition (\ref{eq:wsub_theta_condition}) on $\theta$, which makes this two-step approach necessary.

\begin{lemma}\label{lm:wsub_comparison}
  Let $r_1 \in (0,R)$ and $C \in (0, \Ccrit(r_1))$ with $\Ccrit(r_1)$ as in (\ref{eq:Ccrit_def}). Then there exist $r_\mathrm{min} \in (0, r_1)$, $\theta > 0$ and $\lambda > 1$ such that the following holds for all $r_0 \in [r_\mathrm{min}, r_1)$:

  If
  \begin{equation}\label{eq:wsub_w_0_condition}
    w_0(s) \geq \mu R^n - C (r_1^n - s)^{\frac{m}{m-1}}
  \end{equation}
  for all $s \in [r_0^n, r_1^n]$,
  then there exists $T \in (0,\texist)$ such that
  \begin{equation}\label{eq:wsub_w_comp_result}
    w(s,t) \geq \begin{cases}
      \mu R^n- \lambda C (r_1^n - \theta t - s)^{\frac{m}{m-1}}& \text{if } s < r_1^n - \theta t \\
      \mu R^n &\text{if } s \geq r_1^n - \theta t
    \end{cases}
  \end{equation}
  for all $s \in[r_0^n, r_1^n]$ and a.e.\ $t \in [0, T]$.
\end{lemma}
\begin{proof}
  We begin by fixing $\lambda > 1$ such that 
  \[
    \Asub \defs \lambda C < \Ccrit(r_1).
  \]
  We now fix $r_\mathrm{min} \in (0,r_1)$, $\theta_\mathrm{max} > 0$, $\kappa > 0$ and $\eps_0 \in (0,1)$ according to Lemma~\ref{lm:wsub_construction} and then $\theta \in (0,\theta_\mathrm{max})$ such that (\ref{eq:wsub_theta_condition}) holds for 
  \begin{equation}\label{eq:wsub_def_r2}
    r_2 \defs \left(\frac{r_1^n + R^n}{2}\right)^\frac1n \in (r_1, R).
  \end{equation}
  As (\ref{eq:wsub_w_0_condition}) ensures that 
  \[  
    w_0(r_0^n) \geq \mu R^n - C(r_1^n - r_0^n)^{\frac{m}{m-1}} > \mu R^n - \frac{\lambda + 1}{2}C(r_1^n - r_0^n)^{\frac{m}{m-1}},
  \]
  due to $\frac{\lambda + 1}{2} > 1$, there exist $T \in (0, T_0)$ and $\eps_1 \in (0,\eps_0]$ such that
  \begin{equation}\label{eq:wsub_left_boundary} 
    \we(r_0^n, t) \geq \mu R^n - \frac{\lambda + 1}{2}C(r_1^n - r_0^n)^{\frac{m}{m-1}}
  \end{equation}
  for all $t\in[0,T]$ and $\eps \in (0, \eps_1)$ by Lemma~\ref{lm:uniform_convergence} as well as
  \begin{equation}\label{eq:wsub_delta_small}
    \delta = \left(\frac{\eps\kappa(m-1)}{\Asub m}\right)^{m-1} < \frac{r_1^n - r_0^n}{2}
  \end{equation}
  for all $\eps \in (0, \eps_1)$ and
  \begin{equation}\label{eq:wsub_T_condition}
    T < \max\left\{\frac{r_1^n - r_0^n}{\theta_\mathrm{max}}\left(1 - \left(\frac{2\lambda}{\lambda + 1} \right)^{-\frac{m-1}{m}}\right),\, \frac{r_1^n - r_0^n}{2\theta_\mathrm{max}},\, \frac{R^n - r_1^n}{2\theta_\mathrm{max}}\right\}.
  \end{equation}

  For our first comparison function, let now $\wsub_{\eps, \mathrm{stat}}$ be as in (\ref{eq:wsub_definition}) for all $\eps \in (0,\eps_1)$ with the parameters as fixed above but $\theta$ replaced by $0$. Then we gain
  \[ 
    \wsub_{\eps,\mathrm{stat}}(r_0^n, t) = \mu R^n - \eta - \lambda C(r_1^n - r_0^n)^{\frac{m}{m-1}} \leq \mu R^n - \frac{\lambda + 1}{2}C(r_1^n - r_0^n)^{\frac{m}{m-1}} \leq \we(r_0^n, t) 
  \]
  for all $t \in [0,T]$ and $\eps \in (0, \eps_1)$ due to (\ref{eq:wsub_left_boundary}), $\eta \geq 0$, $\frac{\lambda + 1}{2} < \lambda$ and (\ref{eq:wsub_delta_small}) ensuring that $r_0^n \le r_1^n - \delta = \rho(t) - \delta$. We further directly gain that
  \[
    \wsub_{\eps, \mathrm{stat}}(R^n, t) = \mu R^n = \we(R^n, t)
  \]
  for all $t \in [0,T]$ and $\eps \in (0,\eps_1)$.
  
  Again due to $\eta \geq 0$, $\lambda > 1$ as well as the fact that $w_0(s) \geq \mu R^n - C (r_1^n - s)_+^{\frac{m}{m-1}} \sfed z_0(s)$ for $s \in [r_0^n, R^n]$ by \eqref{eq:wsub_w_0_condition},
  we see that $w_0(s) \geq \wsub_{\eps, \mathrm{stat}}(s,0)$ for all $s\in[r_0 ^n, r_1^n - \delta]\cup[r_1^n, R^n]$.
  Moreover, since $w_0 \ge z_0$ and $z_0(r_1^n) \ge \wsub_{\eps, \mathrm{stat}}(r_1^n,0)$,
  the ordering $w_0(s) \geq \wsub_{\eps, \mathrm{stat}}(s,0)$ in $[r_1^n - \delta, r_1^n]$ and hence on the full interval $[r_0^n, R^n]$
  follows if $\wsub_{\eps, \mathrm{stat}}(\cdot ,0) - z_0$ is increasing in the former interval.
  Indeed,
  \begin{align*}
       \partial_s (\wsub_{\eps, \mathrm{stat}}(s,0) - z_0(s))
    &= \partial_s \left[ -\eps\kappa(R^n - s) + C (r_1^n - s)^\frac{m}{m-1} \right]
    = \eps\kappa - \tfrac{Cm}{m-1} (r_1^n - s)^{\frac{1}{m-1}} \\
    &\geq \eps\kappa - \tfrac{Cm}{m-1} \delta^{\frac{1}{m-1}}
    = \eps \kappa ( 1 - \tfrac{1}{\lambda}) 
  \geq 0
  \end{align*}
  for $s \in [r_1^n - \delta, r_1^n]$ by definition of $\delta$ in (\ref{eq:wsub_delta_small}).
  
  Notably when $\theta = 0$, the condition (\ref{eq:wsub_theta_condition}) is always trivially fulfilled. Thus Lemma~\ref{lm:wsub_construction} ensures that we can apply the comparison result in Lemma~\ref{lm:cmp_princ} to $\we$ and $\wsub_{\eps, \mathrm{stat}}$ and thus gain $\we \geq \wsub_{\eps, \mathrm{stat}}$ on $[r_0^n, R^n]\times[0,T]$. 
  This then directly yields 
  \begin{equation}\label{eq:wsub_right_boundary}
    \we(r_2^n, t) \geq \wsub_{\eps, \mathrm{stat}}(r_2^n, t) = \mu R^n - \eps\kappa \left(\frac{R^n-r_1^n}{2}\right)
  \end{equation}
  for all $t\in [0,T]$ and $\eps \in (0,\eps_1)$.

  For our second comparison argument, we now let
  \[
    \wsub_{\eps, \mathrm{shrink}} \defs \wsub_\eps - \eps\kappa(R^n - r_2^n)
  \] 
  for all $\eps \in (0,\eps_1)$ with $\wsub_\eps$ as in (\ref{eq:wsub_definition}) with the parameters exactly as chosen at the beginning of this proof (including our choice of $\theta$). We now first observe that 
  \begin{align*}
    \wsub_{\eps, \mathrm{shrink}}(r_0^n, t) &= \mu R^n - \eta - \eps\kappa(R^n - r_2^n) - \lambda C(r_1^n - r_0^n - \theta t)^{\frac{m}{m-1}} \\
    &\leq \mu R^n - \frac{\lambda + 1}{2}C(r_1^n - r_0^n)^{\frac{m}{m-1}} \\
    &\leq \we(r_0^n, t)
  \end{align*}
  for all $t\in[0,T]$ and $\eps \in (0,\eps_1)$ due to (\ref{eq:wsub_left_boundary}), (\ref{eq:wsub_T_condition}) and $\eta \geq 0$,
  where we have made use of the fact that the conditions on $\delta$ and $T$ in (\ref{eq:wsub_delta_small}) and (\ref{eq:wsub_T_condition}) entail that $r_0^n \le r_1^n - \theta T - \delta \le \rho(t) - \delta$.   
  Further,
  \begin{align*}
    \wsub_{\eps, \mathrm{shrink}}(r_2^n, t) 
    &= \mu R^n - \eps\kappa(R^n - r_2^n) - \eps\kappa\left(R^n - r_2^n - \theta t\right) \\
    &= \mu R^n - \eps\kappa\left(R^n - r_1^n - \theta t \right) \\
    &\leq \mu R^n - \eps\kappa\left(\frac{R^n - r_1^n}{2}  \right) \leq \we(r_2^n, t)
  \end{align*}
  for all $t \in [0,T]$ and $\eps \in (0,\eps_1)$ due to \eqref{eq:wsub_def_r2}, (\ref{eq:wsub_right_boundary}) and again (\ref{eq:wsub_T_condition}). Plugging this second comparison function into $\Peps$, we then see that 
  \[ 
    \Peps \wsub_{\eps, \mathrm{shrink}} = P_\eps \wsub_\eps + \eps\kappa(R^n - r_2^n)\partial_s \wsub_\eps \leq 0
  \]
  a.e.\ on $[r_0^n, r_2^n]\times[0,T]$ due to Lemma~\ref{lm:wsub_construction} and our previous parameter choices. As initial data ordering follows by the same argument as before, this again allows us to apply Lemma~\ref{lm:cmp_princ} to gain that 
  \[
    \we \geq \wsub_{\eps, \mathrm{shrink}}
  \]
  on $[r_0^n, r_1^n]\times[0,T]$. As $\wsub_{\eps, \mathrm{shrink}}$ converges pointwise to
  \[ 
    (s,t) \mapsto  \begin{cases}
      \mu R^n- \lambda C (r_1^n - \theta t - s)^{\frac{m}{m-1}}& \text{ if } s < r_1^n - \theta t \\
      \mu R^n &\text{ if } s \geq r_1^n - \theta t
    \end{cases}
  \]
  as $\eps \searrow 0$, we then gain our desired result due to the pointwise convergence property in (\ref{eq:w_pw_conv}).
\end{proof}

\subsection{Case of support expansion}
We now treat the support expansion case by showing that $w$ lies above a function of the type seen in (\ref{eq:comp_function_proto}) with a positive coefficient for $\theta$ under an appropriate assumption at $t = 0$. 

In this case, we again move the extension point in our approximate comparison functions slightly to the left of $r_1^n + \theta t$ when compared to (\ref{eq:comp_function_proto}) but also modify said functions in such a fashion as to still allow for a constant extension. This allows us to similarly work around the potential singularity of the second derivative at $r_1^n + \theta t$ as before while ensuring that the extension part is still a supersolution.

\begin{lemma}\label{lm:wsup_construction}
  Let $r_1 \in (0, R)$ and $\Asup > \Ccrit(r_1)$ with $\Ccrit(r_1)$ as in (\ref{eq:Ccrit_def}).
  Then there exists $r_\mathrm{min}\in(0,r_1)$, $\theta > 0$ as well as $\eps_0 \in (0,1)$ such that the following holds for all $r_0 \in [r_\mathrm{min}, r_1)$ and $\eps \in (0, \eps_0)$:

  Let
  \begin{equation}\label{eq:wsup_definition}
    \wsup(s,t) \defs \begin{cases}
      \mu R^n - \Asup(\rfn(t) - s)^{\frac{m}{m-1}} + \eps (\rfn(t) - s) \sfed \wsupmid(s,t) & \text{ if }  s < \rfn(t) - \delta  \\
      \mu R^n + \frac{\delta \eps}{m} \sfed \wsupout(s,t) & \text{ if } s \geq \rfn(t) - \delta 
    \end{cases}
  \end{equation}
  for all $(s,t) \in [r_0^n, R^n]\times[0,\Tsup]$ with $\delta \defs (\frac{\eps(m-1)}{\Asup m})^{m-1}$, $\rho(t) \defs r_1^n + \theta t$ and $\Tsup \defs \frac{R^n - r_1^n}{\theta}$.

  Then $\wsup \in C^1([r_0^n, R^n]\times[0,\Tsup])$ with $\partial_s \wsup \geq 0$ and $\wsup(\cdot,t) \in W_\loc^{2,\infty}((r_0^n, R^n))$ for all $t \in [0, \Tsup]$. Further,
  \begin{equation}\label{eq:wsup_conclusion}
    \Peps \wsup(s,t) \geq 0
  \end{equation}
  for all $s \in [r_0^n, R^n]\setminus\{\rfn(t) - \delta \}$ and $t\in[0,\Tsup]$.
\end{lemma}
\begin{proof}
  We first fix $r_\mathrm{min} \in (0, r_1)$ as well as small $\theta > 0$ such that 
  \begin{equation}\label{eq:wsup_condition_modified}
    \Asup \geq \frac{m-1}{m}\left(\frac{(\mu R^n - \mu r_0^n + 2\theta)(m-1)}{r_0^{2n - 2} n^2}\right)^\frac{1}{m-1}
  \end{equation}
  for all $r_0 \in [r_\mathrm{min}, r_1)$.
  We then further note that the definition of $\Tsup$ entails that
  \[
    r_1^n \leq \rfn(t) \leq R^n
    \qquad \text{for all $t \in [0, \Tsup)$}.
  \] 
  Finally, we let $\eps_0 \defs \min(\frac12, \frac{\theta}{R^n})$ and fix $\eps \in (0, \eps_0)$ as well as $r_0 \in [r_\mathrm{min}, r_1)$.

  By straightforward calculations, we immediately see that 
  \begin{align*}
    \partial_s \wsupmid(s,t) &= \frac{\Asup m(\rfn(t) - s)^{\frac{1}{m-1}}}{m-1} - \eps, \\
    \partial_{ss} \wsupmid(s,t) &= -\frac{\Asup m(\rfn(t) - s)^{\frac{1}{m-1}-1}}{(m-1)^2},  \\
    \partial_t \wsupmid(s,t) &= -\frac{\Asup m \rfn'(t)(\rfn(t) - s)^{\frac{1}{m-1}}}{m-1} + \eps \rfn'(t) = -\theta \partial_s \wsupmid(s,t) 
  \end{align*}
  for all $s\in[r_0^n, \rfn(t) - \delta)$ and $t\in[0,\Tsup]$.
  As the definition of $\delta$ warrants that
  \begin{align*}
    \wsupmid(\rfn(t) - \delta,t) &= \mu R^n - \Asup \delta^{\frac{m}{m-1}} + \eps \delta
    = \mu R^n  + (\eps - \Asup \delta^\frac{1}{m-1})\delta
    %
    %
    %= \mu R^n + \left(1 - \frac{m-1}{m} \right)\delta\eps
    %
    %
    = \mu R^n + \frac{\delta \eps}{m}
    = \wsupout(\rfn(t) - \delta, t)
  \end{align*}
  and 
  \begin{align*}
    \partial_s \wsupmid(\rfn(t) - \delta,t) = \frac{\Asup m}{m-1} \delta^{\frac{1}{m-1}} - \eps = 0 = \partial_s \wsupout(\rfn(t) - \delta, t),
  \end{align*}
  for all $t\in[0,\Tsup]$,
  and as $\partial_{ss}\wsupmid$ is continuous on the compact set $[r_0^n, \rfn(t) - \delta]$ for all $t\in[0,\Tsup]$,
  our desired regularity properties for $\wsup$ are immediately evident.
  Moreover, since $\partial_{ss} \wsupmid(s,t) < 0$, we see that $\partial_{s} \wsupmid(s,t) > 0$ for all $s\in[r_0^n, \rfn(t) - \delta)$ and $t\in[0,\Tsup]$.

  Plugging the above into the parabolic operator $\Peps$, we then see that 
  \[
    \Peps \wsupmid = \partial_s \wsupmid \left[ -\theta - n^2 s^{2-\frac2n}(\partial_s \wsupmid + \eps)^{m-1}\frac{\partial_{ss} \wsupmid}{\partial_s \wsupmid} - \wsupmid + \mu s \right]
  \]
  for all $s\in[r_0^n, \rfn(t) - \delta)$ and $t\in[0,\Tsup]$.
  We now begin by deriving that 
  \begin{align*}
    - \wsupmid + \mu s &= \mu(s - R^n)  + \Asup(\rfn(t) - s)^{\frac{m}{m-1}} - \eps(\rfn(t) - s) \\
    &\geq -\mu (R^n - r_0^n) - \eps \rho(t) \\
    &\geq -\mu (R^n - r_0^n)  - \eps_0 R^n
    \geq -\mu (R^n - r_0^n) - \theta \numberthis \label{eq:wsup_estimate_1}
  \end{align*}
  for all $s\in[r_0^n, \rfn(t) - \delta)$ and $t\in[0,\Tsup]$ by our choice of $\eps_0$.
  Now considering the diffusion term, we further gather that 
  \begin{align*}
    - n^2 s^{2-\frac2n}(\partial_s \wsupmid + \eps)^{m-1}\frac{\partial_{ss} \wsupmid}{\partial_s \wsupmid}
    &= \frac{n^2 s^{2-\frac2n}}{m-1}\left(\frac{\Asup m}{m-1}\right)^{m}\frac{(\rfn(t) - s)^{\frac{1}{m-1}} }{\frac{\Asup m}{m-1}(\rfn(t) - s)^{\frac{1}{m-1}} - \eps} \\
    &\geq \frac{n^2 r_0^{2n-2}}{m-1}\left(\frac{\Asup m}{m-1}\right)^{m-1}
    \geq \mu (R^n - r_0^n) + 2\theta
  \end{align*}
  for all $s\in[r_0^n, \rfn(t) - \delta)$ and $t\in[0,\Tsup]$ due to \eqref{eq:wsup_condition_modified}
  and thus that
  \[ 
    \Peps \wsupmid \geq \partial_s \wsupmid \left[ -\theta + \mu (R^n - r_0^n) + 2\theta - \mu (R^n - r_0^n) - \theta \right] = 0
  \]
  by combining this with (\ref{eq:wsup_estimate_1}).
  Since $\wsupout$ is constant and hence fulfils $\Peps \wsupout = 0$, we obtain \eqref{eq:wsup_conclusion}.
\end{proof}

Using the above family of comparison functions, our desired result then follows from Lemma~\ref{lm:cmp_princ}, which importantly does not require $\wsup$ to be strictly increasing.
\begin{lemma}\label{lm:wsup_comparison}
  Let $r_1 \in (0,R)$ and $C > \Ccrit(r_1)$ with $\Ccrit(r_1)$ as in (\ref{eq:Ccrit_def}). Then there exist $r_\mathrm{min} \in (0, r_1)$, $\theta > 0$ as well as $\lambda > 1$, such that the following holds for all $r_0 \in [r_\mathrm{min},r_1)$:

  If \begin{equation}\label{eq:wsup_w_0_condition}
    w_0(s) \leq \mu R^n - C (r_1^n - s)^{\frac{m}{m-1}}
  \end{equation}
  for all $s \in [r_0^n, r_1^n]$,
  then there exists $T \in (0,\texist)$ such that 
  \begin{equation}\label{eq:wsup_w_comp_result}
    w(s,t) \leq \begin{cases}
      \mu R^n- \frac{C}{\lambda} (r_1^n + \theta t - s)^{\frac{m}{m-1}}& \text{ if } s < r_1^n + \theta t, \\
      \mu R^n &\text{ if } s \geq r_1^n + \theta t
    \end{cases}
  \end{equation}
  for all $s \in[r_0^n, R^n]$ and a.e.\ $t \in [0, T]$.
\end{lemma}
\begin{proof}
  We begin by fixing $\lambda > 1$ such that 
  \[
    \Asup \defs \frac{C}{\lambda} > \Ccrit(r_1).
  \]
  We now let $\Tsup \defs \frac{R^n-r_1^n}{\theta} > 0$ and fix $r_\mathrm{min} \in (0,r_1)$, $\theta > 0$ as well as $\eps_0 \in (0,1)$ according to Lemma~\ref{lm:wsup_construction}.
  As (\ref{eq:wsup_w_0_condition}) ensures that
  \[ 
    w_0(r_0^n) \leq \mu R^n - C(r_1^n - r_0^n)^{\frac{m}{m-1}} < \mu R^n - \frac{2C}{\lambda + 1}(r_1^n - r_0^n)^{\frac{m}{m-1}}
  \]
  due to $\frac{\lambda + 1}{2} > 1$, Lemma~\ref{lm:uniform_convergence} allows us to fix $T \in (0,\min(T_0,\Tsup) )$ and $\eps_1 \in (0,\eps_0]$ such that 
  \begin{equation}\label{eq:wsup_left_boundary}
    \we(r_0^n,t) \leq \mu R^n - \frac{2C}{\lambda+1}(r_1^n - r_0^n)^{\frac{m}{m-1}}
  \end{equation}
  for all $t\in [0,T]$ and $\eps \in (0,\eps_1)$ as well as 
  \begin{equation}\label{eq:wsup_delta_small}
    \delta = \left(\frac{\eps(m-1)}{\Asup m}\right)^{m-1} < r_1^n - r_0^n 
  \end{equation}
  for all $\eps \in (0, \eps_1)$ and
  \begin{equation}\label{eq:wsup_T_condition}
    T \leq \frac{r_1^n - r_0^n}{\theta}\left[\left( \frac{2\lambda}{\lambda + 1} \right)^\frac{m-1}{m} - 1\right].
  \end{equation}
  Let now $\wsup_\eps$ with $\Peps \wsup_\eps \geq 0$ be as in (\ref{eq:wsup_definition}) for all $\eps \in (0,\eps_1)$ with the parameters as chosen above. We then check the necessary ordering at the left boundary as follows:
  \begin{align*}
    \wsup_\eps(r_0^n, t) &= \mu R^n - \frac{C}{\lambda}(r_1^n - r_0^n + \theta t)^{\frac{m}{m-1}} + \eps(r_1^n - r_0^n + \theta t) \\
    &\geq \mu R^n - \frac{2C}{\lambda+1}(r_1^n - r_0^n)^{\frac{m}{m-1}} \geq \we(r_0^n, t)
  \end{align*}
  for all $t \in [0,T]$ and $\eps \in (0,\eps_1)$ due to (\ref{eq:wsup_left_boundary}) and (\ref{eq:wsup_T_condition}) as well as (\ref{eq:wsup_delta_small}) ensuring that $r_0^n < r_1^n - \delta \leq \rho(t) - \delta$. Regarding the right boundary point, we observe that 
  \[
    \wsup_\eps(R^n, t) = \mu R^n + \frac{\delta\eps}{m} \geq \mu R^n
  \]
  for all $t \in [0,T]$ and $\eps \in (0,\eps_1)$, where $T \leq \Tsup = \frac{R^n-r_1^n}{\theta}$ ensures that $R^n \geq \rho(t) > \rho(t) - \delta$.
  Moreover, we can conclude that 
  \[
    \wsup_\eps(s, 0) = \mu R^n - \frac{C}{\lambda}(r_1^n - s)^{\frac{m}{m-1}} + \eps(r_1^n - s) \geq \mu R^n - C(r_1^n - s)^{\frac{m}{m-1}} \geq w_0(s)
  \]
  for all $s \in [r_0^n, r_1^n-\delta)$ due to $\lambda > 1$ and (\ref{eq:wsup_w_0_condition}) as well as 
  \[
    \wsup_\eps(s, 0) = \mu R^n + \frac{\delta \eps}{m} \geq \mu R^n\geq w_0(s)
  \]
  for all $s \in [r_1^n-\delta, R^n]$ since by construction $w_0 \leq \mu R^n$.
  Thus, we can now employ Lemma~\ref{lm:cmp_princ} to gain that 
  \[
    \we(s,t)\leq\wsup_\eps(s,t)
  \]
  for all $(s,t) \in [r_0^n, R^n] \times [0,T]$. Using that $\wsup_\eps$ converges pointwise to 
  \[
    (s,t) \mapsto \begin{cases}
      \mu R^n - \frac{C}{\lambda}(r_1^n + \theta t - s)^\frac{m}{m-1} &\text{if } s < \theta t + r_1^n, \\
      \mu R^n & \text{if } s \geq \theta t + r_1^n
    \end{cases}
  \]
  as $\eps \searrow 0$, we then gain our desired result due to the pointwise convergence property in (\ref{eq:w_pw_conv}).
\end{proof}

\section{Proof of Theorem~\ref{th:main}}\label{sec:proof_main_thm}
Having now established all necessary prerequisites, we can start putting the puzzle pieces together to prove our main theorem.
\begin{proof}[Proof of Theorem~\ref{th:main}]
  We begin by treating the support shrinking case. We thus now assume that (\ref{eq:cond_shrinking}) holds for some positive $A < \Acrit$ and $r_0 \in (0,r_1)$.
  The definitions of $\Acrit$ and $\Ccrit$ in \eqref{eq:def_c_crit} and (\ref{eq:Ccrit_def}) entail 
  \[
    \Acrit \frac{n^{-\frac{1}{m-1}}r_1^{-\frac{(n-1)m}{m-1}} r_1^{n-1}(m-1)}{m} = \Ccrit(r_1),
  \]
  so that by replacing $r_0$ with an $r_0$ sufficiently close to $r_1$, if necessary, we may assume
  \[
    C \defs \frac{A n^{-\frac{1}{m-1}}r_0^{-\frac{(n-1)m}{m-1}} r_1^{n-1}(m-1)}{m} < \Ccrit(r_1).
  \]
  Let now $r_\mathrm{min} \in (0,r_1)$ be as provided by Lemma~\ref{lm:wsub_comparison} for the constant $C$ fixed above. Let then further $
    r_\star \defs \max(r_\mathrm{min}, r_0)
  $.
  Lemma~\ref{lm:crit_w0} then yields that \[
    w_0(s) \geq \mu R^n - C(r_1^n - s)^\frac{m}{m-1}
  \]
  for all $s \in (r_\star^n, r_1^n)$, which in turn allows us to apply Lemma~\ref{lm:wsub_comparison}  to further fix $T \in (0,\texist)$, $\lambda > 0$ and $\theta > 0$ such that (\ref{eq:wsub_w_comp_result}) holds for all $s \in [r_\star^n, r_1^n]$ and a.e.\ $t \in [0,T]$. If we complement this lower bound with the fact that $w(s,t) \leq \mu R^n$ for all $s \in [r_\star^n, R^n]$ and a.e.\ $t \in [0,T]$ due to (\ref{eq:w_upper_bound}), we gain that 
  \[
    w(s,t) = \mu R^n
  \]
  for all $s \in [r_1^n - \theta t, R^n]$ and a.e.\ $t\in[0,T]$. This implies that 
  \[
    0 = w(R^n, t) - w(r_1^n - \theta t, t) = n\int^{R}_{(r_1^n - \theta t)^\frac{1}{n}} \rho^{n-1} u(\rho,t) \drho
  \]
  for a.e.\ $t \in [0,T]$. Given the a.e.\ nonnegativity of $u$ and positivity of $\rho$, it thus follows that 
  \[
    u(s,t) = 0
  \]
  for a.e.\ $s \in [(r_1^n - \theta t)^\frac{1}{n}, R]$ and a.e.\ $t\in[0, T]$ and hence
  \[
    \sup (\ess\supp u(\cdot, t)) \leq (r_1^n - \theta t)^\frac{1}{n}
    \qquad \text{for a.e.\ $t \in (0,T)$}.
  \]
  As here $(r_1^n - \theta t)^\frac{1}{n} - r_1 \le -\frac1n r_1^{1-n} \theta t \sfed -\zeta t$ for all $t \in (0, T)$ by the mean value theorem, this implies \eqref{eq:result_shrinking}.

  By combining Lemma~\ref{lm:crit_w0} and Lemma~\ref{lm:wsup_comparison} while assuming (\ref{eq:cond_expanding}), in a very similar fashion to the argument above, we can gain another set of $C  > 0$, $r_\star\in(0,r_1)$, $T \in (0,\texist)$, $\lambda > 0$ and $\theta > 0$ such that (\ref{eq:wsup_w_comp_result}) holds for all $s \in [r_\star^n, R^n]$ and a.e.\ $t \in [0, T]$. This implies that 
  \[
    w(R^n, t) - w(s,t) \geq \mu R^n - \mu R^n + \frac{C}{\lambda} (r_1^n + \theta t - s)^\frac{m}{m-1} > 0
  \]
  and thus 
  \[
    n\int^{R}_{s^\frac{1}{n}} \rho^{n-1} u(\rho,t) \drho > 0
  \]
  for all $s \in (r_\star^n, r_1^n + \theta t)$ and a.e.\ $t \in [0,T]$.  Therefore, for all $s \in (r_\star^n, r_1^n + \theta t)$ and a.e.\ $t \in [0,T]$, there must exist a set $M(s, t) \subseteq (s^\frac{1}{n}, R)$ of positive measure such that $u(\cdot, t) > 0$ on $M(s, t)$. This directly implies $\sup(\ess\supp u(\cdot, t)) \geq s^\frac1n$ for all $s \in (r_\star^n, r_1^n + \theta t)$ and a.e.\ $t \in [0,T]$ and hence
  \[
    \sup (\ess\supp u(\cdot, t)) \geq (r_1^n + \theta t)^\frac{1}{n}
    \qquad \text{for a.e.\ $t \in (0,T)$}.
  \]
  By a final application of the mean value theorem,
  we conclude that \eqref{eq:result_expanding} holds for $\zeta \defs \frac1n (r_1^n + \theta T)^{\frac1n-1} \theta$.
\end{proof} 

\section*{Acknowledgments}
The second author acknowledges support of the \emph{Deutsche Forschungsgemeinschaft} in the context of the project \emph{Fine structures in interpolation inequalities and application to parabolic problems}, project number 462888149.
 
\addcontentsline{toc}{section}{References}
%\bibliography{bib}

\footnotesize
\begin{thebibliography}{10}
\setlength{\itemsep}{0.2pt}

\bibitem{AlikakosPointwiseBehaviorSolutions1985}
\textsc{Alikakos, N.~D.}:
\newblock {\em On the pointwise behavior of the solutions of the porous medium
  equation as {{t}} approaches zero or infinity}.
\newblock Nonlinear Anal.,
  \href{https://doi.org/10.1016/0362-546X(85)90088-4}{9(10):1095--1113}, 1985.
\newblock

\bibitem{BellomoEtAlMathematicalTheoryKeller2015}
\textsc{Bellomo, N.}, \textsc{Bellouquid, A.}, \textsc{Tao, Y.}, and
  \textsc{Winkler, M.}:
\newblock {\em Toward a mathematical theory of {{Keller}}--{{Segel}} models of
  pattern formation in biological tissues}.
\newblock Math. Models Methods Appl. Sci.,
  \href{https://doi.org/10.1142/S021820251550044X}{25(09):1663--1763}, 2015.
\newblock

\bibitem{bellomoFinitetimeBlowdegenerateChemotaxis2017}
\textsc{Bellomo, N.} and \textsc{Winkler, M.}:
\newblock {\em Finite-time blow-up in a degenerate chemotaxis system with flux
  limitation}.
\newblock Trans. Am. Math. Soc. Ser. B,
  \href{https://doi.org/10.1090/btran/17}{4:31--67}, 2017.
\newblock

\bibitem{BlackAbsenceDeadcoreFormations2024}
\textsc{Black, T.}:
\newblock {\em Absence of dead-core formations in chemotaxis systems with
  degenerate diffusion}.
\newblock Appl. Math. Lett.,
  \href{https://doi.org/10.1016/j.aml.2024.109361}{161:Paper No. 109361},
  2025.
\newblock

\bibitem{BlackRefiningHolderRegularity2024}
\textsc{Black, T.}:
\newblock {\em Refining {{H{\"o}lder}} regularity theory in degenerate
  drift-diffusion equations}.
\newblock Preprint, \href{https://arxiv.org/abs/2410.03307}{ arXiv:2410.03307},
  2024.

\bibitem{BlanchetLaurencotParabolicparabolicKellerSegelSystem2013}
\textsc{Blanchet, A.} and \textsc{Lauren{\c c}ot, {\relax Ph}.}:
\newblock {\em The parabolic-parabolic {{Keller-Segel}} system with critical
  diffusion as a gradient flow in {$\mathbb R^n, d \geq 3$}}.
\newblock Comm. Partial Differ. Equ.,
  \href{https://doi.org/10.1080/03605302.2012.757705}{38(4):658--686}, 2013.
\newblock

\bibitem{BurgerEtAlKellerSegelModelChemotaxis2006}
\textsc{Burger, M.}, \textsc{Di~Francesco, M.}, and \textsc{{Dolak-Struss},
  Y.}:
\newblock {\em The {{Keller-Segel}} model for chemotaxis with prevention of
  overcrowding: Linear vs.nonlinear diffusion}.
\newblock SIAM J. Math. Anal.,
  \href{https://doi.org/10.1137/050637923}{38(4):1288--1315}, 2006.
\newblock

\bibitem{CaoFuestFinitetimeBlowfullyParabolic2024}
\textsc{Cao, X.} and \textsc{Fuest, M.}:
\newblock {\em Finite-time blow-up in fully parabolic quasilinear
  {{Keller}}--{{Segel}} systems with supercritical exponents}.
\newblock Calc. Var. Partial Differ. Equ.,
  \href{https://doi.org/0.1007/s00526-025-02944-4}{64:Art. 89}, 2025.
\newblock

\bibitem{CieslakStinnerFinitetimeBlowupGlobalintime2012}
\textsc{Cie{\'s}lak, T.} and \textsc{Stinner, {\relax Ch}.}:
\newblock {\em Finite-time blowup and global-in-time unbounded solutions to a
  parabolic--parabolic quasilinear {{Keller}}--{{Segel}} system in higher
  dimensions}.
\newblock J. Differ. Equ.,
  \href{https://doi.org/10.1016/j.jde.2012.01.045}{252(10):5832--5851}, 2012.
\newblock

\bibitem{CieslakStinnerFiniteTimeBlowupSupercritical2014}
\textsc{Cie{\'s}lak, T.} and \textsc{Stinner, {\relax Ch}.}:
\newblock {\em Finite-time blowup in a supercritical quasilinear
  parabolic-parabolic {{Keller}}--{{Segel}} system in dimension 2}.
\newblock Acta Appl. Math.,
  \href{https://doi.org/10.1007/s10440-013-9832-5}{129(1):135--146}, 2014.
\newblock

\bibitem{CieslakStinnerNewCriticalExponents2015}
\textsc{Cie{\'s}lak, T.} and \textsc{Stinner, {\relax Ch}.}:
\newblock {\em New critical exponents in a fully parabolic quasilinear
  {{Keller}}--{{Segel}} system and applications to volume filling models}.
\newblock J. Differ. Equ.,
  \href{https://doi.org/10.1016/j.jde.2014.12.004}{258(6):2080--2113}, 2015.
\newblock

\bibitem{CieslakWinklerFinitetimeBlowupQuasilinear2008}
\textsc{Cie{\'s}lak, T.} and \textsc{Winkler, M.}:
\newblock {\em Finite-time blow-up in a quasilinear system of chemotaxis}.
\newblock Nonlinearity,
  \href{https://doi.org/10.1088/0951-7715/21/5/009}{21(5):1057--1076}, 2008.
\newblock

\bibitem{DalPassoEtAlWaitingTimePhenomena2003}
\textsc{Dal~Passo, R.}, \textsc{Giacomelli, L.}, and \textsc{Gr{\"u}n, G.}:
\newblock {\em Waiting time phenomena for degenerate parabolic equations---a
  unifying approach}.
\newblock In {\em Geometric Analysis and Nonlinear Partial Differential
  Equations}, pages 637--648. Springer, Berlin, 2003.

\bibitem{FischerAdvectiondrivenSupportShrinking2013}
\textsc{Fischer, J.}:
\newblock {\em Advection-driven support shrinking in a chemotaxis model with
  degenerate mobility}.
\newblock SIAM J. Math. Anal.,
  \href{https://doi.org/10.1137/120874291}{45(3):1585--1615}, 2013.
\newblock

\bibitem{GurtinMacCamyDiffusionBiologicalPopulations1977}
\textsc{Gurtin, M.~E.} and \textsc{MacCamy, R.~C.}:
\newblock {\em On the diffusion of biological populations}.
\newblock Math. Biosci.,
  \href{https://doi.org/10.1016/0025-5564(77)90062-1}{33(1-2):35--49}, 1977.
\newblock

\bibitem{HashiraEtAlFinitetimeBlowupQuasilinear2018}
\textsc{Hashira, T.}, \textsc{Ishida, S.}, and \textsc{Yokota, T.}:
\newblock {\em Finite-time blow-up for quasilinear degenerate
  {{Keller}}--{{Segel}} systems of parabolic--parabolic type}.
\newblock J. Differ. Equ.,
  \href{https://doi.org/10.1016/j.jde.2018.01.038}{264(10):6459--6485}, 2018.
\newblock

\bibitem{HillenPainterUserGuidePDE2009}
\textsc{Hillen, T.} and \textsc{Painter, K.~J.}:
\newblock {\em A user's guide to {{PDE}} models for chemotaxis}.
\newblock J. Math. Biol.,
  \href{https://doi.org/10.1007/s00285-008-0201-3}{58(1-2):183--217}, 2009.
\newblock

\bibitem{HorstmannWinklerBoundednessVsBlowup2005}
\textsc{Horstmann, D.} and \textsc{Winkler, M.}:
\newblock {\em Boundedness vs. blow-up in a chemotaxis system}.
\newblock J. Differ. Equ.,
  \href{https://doi.org/10.1016/j.jde.2004.10.022}{215(1):52--107}, 2005.
\newblock

\bibitem{IshidaEtAlBoundednessQuasilinearKeller2014}
\textsc{Ishida, S.}, \textsc{Seki, K.}, and \textsc{Yokota, T.}:
\newblock {\em Boundedness in quasilinear {{Keller}}--{{Segel}} systems of
  parabolic--parabolic type on non-convex bounded domains}.
\newblock J. Differ. Equ.,
  \href{https://doi.org/10.1016/j.jde.2014.01.028}{256(8):2993--3010}, 2014.
\newblock

\bibitem{IshidaYokotaBlowupFiniteInfinite2013}
\textsc{Ishida, S.} and \textsc{Yokota, T.}:
\newblock {\em Blow-up in finite or infinite time for quasilinear degenerate
  {{Keller-Segel}} systems of parabolic-parabolic type}.
\newblock Discrete Contin. Dyn. Syst. Ser. B,
  \href{https://doi.org/10.3934/dcdsb.2013.18.2569}{18(10):2569--2596}, 2013.
\newblock

\bibitem{JagerLuckhausExplosionsSolutionsSystem1992}
\textsc{J{\"a}ger, W.} and \textsc{Luckhaus, S.}:
\newblock {\em On explosions of solutions to a system of partial differential
  equations modelling chemotaxis}.
\newblock Trans. Am. Math. Soc.,
  \href{https://doi.org/10.2307/2153966}{329(2):819--824}, 1992.
\newblock

\bibitem{KellerSegelInitiationSlimeMold1970}
\textsc{Keller, E.~F.} and \textsc{Segel, L.~A.}:
\newblock {\em Initiation of slime mold aggregation viewed as an instability}.
\newblock J. Theor. Biol.,
  \href{https://doi.org/10.1016/0022-5193(70)90092-5}{26(3):399--415}, 1970.
\newblock

\bibitem{KimYaoPatlakKellerSegelModelIts2012}
\textsc{Kim, I.} and \textsc{Yao, Y.}:
\newblock {\em The {{Patlak-Keller-Segel}} model and its variations: Properties
  of solutions via maximum principle}.
\newblock SIAM J. Math. Anal.,
  \href{https://doi.org/10.1137/110823584}{44(2):568--602}, 2012.
\newblock

\bibitem{KowalczykSzymanskaGlobalExistenceSolutions2008}
\textsc{Kowalczyk, R.} and \textsc{Szyma{\'n}ska, Z.}:
\newblock {\em On the global existence of solutions to an aggregation model}.
\newblock J. Math. Anal. Appl.,
  \href{https://doi.org/10.1016/j.jmaa.2008.01.005}{343(1):379--398}, 2008.
\newblock

\bibitem{LankeitWinklerFacingLowRegularity2019}
\textsc{Lankeit, J.} and \textsc{Winkler, M.}:
\newblock {\em Facing low regularity in chemotaxis systems}.
\newblock Jahresber. Dtsch. Math.-Ver.,
  \href{https://doi.org/10.1365/s13291-019-00210-z}{122:35--64}, 2019.
\newblock

\bibitem{LaurencotMatiocFiniteSpeedPropagation2017}
\textsc{Lauren{\c c}ot, {\relax Ph}.} and \textsc{Matioc, B.-V.}:
\newblock {\em Finite speed of propagation and waiting time for a thin-film
  {{Muskat}} problem}.
\newblock Proc. Roy. Soc. Edinburgh Sect. A,
  \href{https://doi.org/10.1017/S030821051600038X}{147(4):813--830}, 2017.
\newblock

\bibitem{LaurencotMizoguchiFiniteTimeBlowup2017}
\textsc{Lauren{\c c}ot, {\relax Ph}.} and \textsc{Mizoguchi, N.}:
\newblock {\em Finite time blowup for the parabolic--parabolic
  {{Keller}}--{{Segel}} system with critical diffusion}.
\newblock Annales de l'Institut Henri Poincar{\'e} C, Analyse non lin{\'e}aire,
  \href{https://doi.org/10.1016/j.anihpc.2015.11.002}{34(1):197--220}, 2017.
\newblock

\bibitem{PainterHillenVolumefillingQuorumsensingModels2002}
\textsc{Painter, K.} and \textsc{Hillen, T.}:
\newblock {\em Volume-filling and quorum-sensing in models for chemosensitive
  movement}.
\newblock Can. Appl. Math. Q., 10(4):501--544, 2002.

\bibitem{PorzioVespriHolderEstimatesLocal1993}
\textsc{Porzio, M.} and \textsc{Vespri, V.}:
\newblock {\em Holder estimates for local solutions of some doubly nonlinear
  degenerate parabolic equations}.
\newblock J. Differ. Equ.,
  \href{https://doi.org/10.1006/jdeq.1993.1045}{103(1):146--178}, 1993.
\newblock

\bibitem{StevensWinklerTaxisdrivenPersistentLocalization2022}
\textsc{Stevens, A.} and \textsc{Winkler, M.}:
\newblock {\em Taxis-driven persistent localization in a degenerate
  {{Keller-Segel}} system}.
\newblock Comm. Partial Differ. Equ.,
  \href{https://doi.org/10.1080/03605302.2022.2122836}{47(12):2341--2362},
  2022.
\newblock

\bibitem{SugiyamaGlobalExistenceSubcritical2006}
\textsc{Sugiyama, Y.}:
\newblock {\em Global existence in sub-critical cases and finite time blow-up
  in super-critical cases to degenerate {{Keller-Segel}} systems}.
\newblock Differ. Integral Equ., 19(8):841--876, 2006.

\bibitem{SugiyamaFiniteSpeedPropagation2012}
\textsc{Sugiyama, Y.}:
\newblock {\em Finite speed of propagation in 1-{{D}} degenerate
  {{Keller-Segel}} system}.
\newblock Math. Nachrichten,
  \href{https://doi.org/10.1002/mana.200810258}{285(5-6):744--757}, 2012.
\newblock

\bibitem{TaoWinklerBoundednessQuasilinearParabolic2012}
\textsc{Tao, Y.} and \textsc{Winkler, M.}:
\newblock {\em Boundedness in a quasilinear parabolic--parabolic
  {{Keller}}--{{Segel}} system with subcritical sensitivity}.
\newblock J. Differ. Equ.,
  \href{https://doi.org/10.1016/j.jde.2011.08.019}{252(1):692--715}, 2012.
\newblock

\bibitem{VazquezPorousMediumEquation2006}
\textsc{Vazquez, J.~L.}:
\newblock {\em The {{Porous Medium Equation}}}.
\newblock Oxford University Press, 2006.
\newblock

\bibitem{WinklerDoesVolumefillingEffect2009}
\textsc{Winkler, M.}:
\newblock {\em Does a `volume-filling effect' always prevent chemotactic
  collapse?}
\newblock Math. Methods Appl. Sci.,
  \href{https://doi.org/10.1002/mma.1146}{33(1):12--24}, 2009.
\newblock

\bibitem{XuEtAlChemotaxisModelDegenerate2020}
\textsc{Xu, T.}, \textsc{Ji, S.}, \textsc{Mei, M.}, and \textsc{Yin, J.}:
\newblock {\em On a chemotaxis model with degenerate diffusion: {{Initial}}
  shrinking, eventual smoothness and expanding}.
\newblock J. Differ. Equ.,
  \href{https://doi.org/10.1016/j.jde.2019.08.013}{268(2):414--446}, 2020.
\newblock

\end{thebibliography}

\end{document}